%% file: Master-arXiv.tex
\documentclass[a4paper,oneside,english,italian,british]{amsart}
\usepackage[T1]{fontenc}
\usepackage[latin9]{inputenc}
\usepackage{color}
\usepackage{babel}
\usepackage{array}
\usepackage{prettyref}
\usepackage{float}
\usepackage{multirow}
\usepackage{amsthm}
\usepackage{amstext}
\usepackage{amssymb}
\usepackage{subscript}
\usepackage[unicode=true,pdfusetitle,
 bookmarks=true,bookmarksnumbered=true,bookmarksopen=false,
 breaklinks=false,pdfborder={0 0 0},backref=false,colorlinks=true]
 {hyperref}

\makeatletter

\pdfpageheight\paperheight
\pdfpagewidth\paperwidth

\providecommand{\tabularnewline}{\\}

\numberwithin{equation}{section}
\numberwithin{figure}{section}
\numberwithin{table}{section}
\usepackage{enumitem}		
\theoremstyle{plain}
\newtheorem{thm}{\protect\theoremname}[section]
  \theoremstyle{definition}
  \newtheorem{defn}[thm]{\protect\definitionname}
  \theoremstyle{remark}
  \newtheorem{notation}[thm]{\protect\notationname}
  \theoremstyle{plain}
  \newtheorem{fact}[thm]{\protect\factname}
  \theoremstyle{plain}
  \newtheorem{prop}[thm]{\protect\propositionname}
  \theoremstyle{remark}
  \newtheorem{rem}[thm]{\protect\remarkname}
  \theoremstyle{plain}
  \newtheorem{cor}[thm]{\protect\corollaryname}
  \theoremstyle{plain}
  \newtheorem{lem}[thm]{\protect\lemmaname}

\newrefformat{prop}{\protect\propositionname\ \ref{#1}}
\newrefformat{fact}{\protect\factname\ \ref{#1}}
\newrefformat{cor}{\protect\corollaryname\ \ref{#1}}
\newrefformat{lem}{\protect\lemmaname\ \ref{#1}}
\newrefformat{rmk}{\protect\remarkname\ \ref{#1}}
\newrefformat{sub}{\protect{subsection}\ \ref{#1}}

\DeclareRobustCommand{\SkipTocEntry}[5]{}

\makeatother

  \addto\captionsbritish{\renewcommand{\corollaryname}{Corollary}}
  \addto\captionsbritish{\renewcommand{\definitionname}{Definition}}
  \addto\captionsbritish{\renewcommand{\factname}{Fact}}
  \addto\captionsbritish{\renewcommand{\lemmaname}{Lemma}}
  \addto\captionsbritish{\renewcommand{\notationname}{Notation}}
  \addto\captionsbritish{\renewcommand{\propositionname}{Proposition}}
  \addto\captionsbritish{\renewcommand{\remarkname}{Remark}}
  \addto\captionsbritish{\renewcommand{\theoremname}{Theorem}}
  \addto\captionsenglish{\renewcommand{\corollaryname}{Corollary}}
  \addto\captionsenglish{\renewcommand{\definitionname}{Definition}}
  \addto\captionsenglish{\renewcommand{\factname}{Fact}}
  \addto\captionsenglish{\renewcommand{\lemmaname}{Lemma}}
  \addto\captionsenglish{\renewcommand{\notationname}{Notation}}
  \addto\captionsenglish{\renewcommand{\propositionname}{Proposition}}
  \addto\captionsenglish{\renewcommand{\remarkname}{Remark}}
  \addto\captionsenglish{\renewcommand{\theoremname}{Theorem}}
  \addto\captionsitalian{\renewcommand{\corollaryname}{Corollario}}
  \addto\captionsitalian{\renewcommand{\definitionname}{Definizione}}
  \addto\captionsitalian{\renewcommand{\factname}{Fatto}}
  \addto\captionsitalian{\renewcommand{\lemmaname}{Lemma}}
  \addto\captionsitalian{\renewcommand{\notationname}{Notazione}}
  \addto\captionsitalian{\renewcommand{\propositionname}{Proposizione}}
  \addto\captionsitalian{\renewcommand{\remarkname}{Osservazione}}
  \addto\captionsitalian{\renewcommand{\theoremname}{Teorema}}
  \providecommand{\corollaryname}{Corollary}
  \providecommand{\definitionname}{Definition}
  \providecommand{\factname}{Fact}
  \providecommand{\lemmaname}{Lemma}
  \providecommand{\notationname}{Notation}
  \providecommand{\propositionname}{Proposition}
  \providecommand{\remarkname}{Remark}
\providecommand{\theoremname}{Theorem}

\begin{document}
\input{Macros.tex}

\global\long\def\intv#1#2{\left\langle #1;#2\right\rangle }

\global\long\def\zEEz#1#2{\intv{#2}{#1\left(#2\right)}}

\global\long\def\zEz#1{\zEEz E{#1}}

\global\long\def\span{\mathrm{span}}

\global\long\def\im{\mathrm{im}}

\global\long\def\odim{o\mathrm{-dim}}

\title{Involutions on Zilber fields}

\selectlanguage{italian}%

\author{Vincenzo Mantova}

\selectlanguage{british}%

\date{Draft of 25th May 2013}


\begin{abstract}
After recalling the definition of Zilber fields, and the main conjecture
behind them, we prove that Zilber fields of cardinality up to the
continuum have involutions, i.e., automorphisms of order two analogous
to complex conjugation on $\C_{\exp}$. Moreover, we also prove that
for continuum cardinality there is an involution whose fixed field,
as a real closed field, is isomorphic to the field of real numbers,
and such that the kernel is exactly $2\pi i\Z$, answering a question
of Zilber, Kirby, Macintyre and Onshuus.

The proof is obtained with an explicit construction of a Zilber field
with the required properties. As further applications of this technique,
we also classify the exponential subfields of Zilber fields, and we
produce some exponential fields with involutions such that the exponential
function is order-preserving, or even continuous, and all of the axioms
of Zilber fields are satisfied except for the strong exponential-algebraic
closure, which gets replaced by some weaker axioms.
\end{abstract}
\maketitle
\tableofcontents{}

\input{Introduction.tex}

\input{DefAxioms.tex}

\input{Construction.tex}

\input{Involuntary.tex}

\input{Facts.tex}

\input{CCP.tex}

\input{Roots.tex}

\input{GRestriction.tex}

\input{Rotundity.tex}

\input{RootsGRes.tex}

\input{Proof.tex}

\input{Further.tex}

\addtocontents{toc}{\SkipTocEntry}

\input{Master.bbl}
\end{document}

%% file: Macros.tex
\selectlanguage{english}%

\global\long\def\A{\mathbb{A}}

\global\long\def\B{\mathbb{B}}

\global\long\def\C{\mathbb{C}}

\global\long\def\F{\mathbb{F}}

\global\long\def\G{\mathbb{G}}

\global\long\def\N{\mathbb{N}}

\global\long\def\P{\mathbb{P}}

\global\long\def\Q{\mathbb{Q}}

\global\long\def\oQ{\overline{\mathbb{Q}}}

\global\long\def\R{\mathbb{R}}

\global\long\def\Sb{\mathbb{S}}

\global\long\def\Z{\mathbb{Z}}

\global\long\def\Ac{\mathcal{A}}

\global\long\def\Bc{\mathcal{B}}

\global\long\def\Cc{\mathcal{C}}

\global\long\def\Jc{\mathcal{J}}

\global\long\def\Lc{\mathcal{L}}

\global\long\def\Mc{\mathcal{M}}

\global\long\def\Oc{\mathcal{O}}

\global\long\def\Rc{\mathcal{R}}

\global\long\def\Sc{\mathcal{S}}

\global\long\def\Uc{\mathcal{U}}

\global\long\def\Vc{\mathcal{V}}

\global\long\def\Wc{\mathcal{W}}

\global\long\def\Aut{\mathrm{Aut}}

\global\long\def\acl{\mathrm{acl}}

\global\long\def\cl{\mathrm{cl}}

\global\long\def\divs{\mathrm{div}}

\global\long\def\dom{\mathrm{dom}}

\global\long\def\Gal{\mathrm{Gal}}

\global\long\def\ld{\mathrm{lin.d.}}

\global\long\def\ord{\mathrm{ord}}

\global\long\def\preim{\mathrm{preim}}

\global\long\def\ran{\mathrm{ran}}

\global\long\def\rank{\mathrm{rank}}

\global\long\def\supp{\mathrm{supp}}

\global\long\def\td{\mathrm{tr.deg.}}

\global\long\def\st{^{*}\!}
\selectlanguage{british}%

%% file: Introduction.tex
\section{Introduction}

Zilber fields are a recent creation of Boris Zilber~\cite{Zilber2005}
born during the study of the model theory of exponential maps, and
mainly of $\C_{\exp}$. Indeed, a Zilber field is first of all a structure
$(K,0,1,+,\cdot,E)$, where $K$ is a field and $E$ is an exponential
function, i.e., a map satisfying 
\[
E(x+y)=E(x)\cdot E(y).
\]
A certain number of axioms must hold for the structure to be a Zilber
field. In analogy with the classical cases of $\C_{\exp}$ and $\R_{\exp}$,
we will denote by ``$K_{E}$'' the structure made up by a field
$K$ and an exponential function $E$ in the above sense. The function
$E$ of a Zilber field is called \emph{pseudoexponentiation}.

The axioms of Zilber fields state that $K_{E}$ must be somewhat similar
to the structure $\C_{\exp}$, but also force it to have some properties
that are deep conjectures for the classical exponential function.
Schanuel's Conjecture, rephrased with $E$ in place of $\exp$, is
included among these axioms.

The remarkable result by Zilber~\cite{Zilber2005}, followed by a
fix in \cite{Bays2012,Bays2013}, is that these axioms are expressible
in a reasonably simple way in a suitable infinitary language, and,
more importantly, they are \emph{uncountably categorical}. In Zilber's
philosophy, categorical structures should correspond to natural mathematical
objects, so he conjectured that $\C_{\exp}$ is exactly the model
of cardinality $2^{\aleph_{0}}$. The unique model of cardinality
$2^{\aleph_{0}}$ has been called ``$\B$'' by Macintyre, after
Boris Zilber. Here we will use the variant ``$\B_{E}$'' when we
want to explicitly name the pseudoexponential function.

It is quite natural to ask which properties of $\C_{\exp}$ we are
able to prove on Zilber fields, and vice versa, in the hope of either
supporting or refuting the conjecture. The recent results in this
direction on $\B$ \cite{D'Aquino2010}, \cite{Shkop2011} and on
$\C_{\exp}$ \cite{Zilber2011a} support the conjecture.

One question in particular is whether $\B$ has an involution, i.e.,
an automorphism of order two, as complex conjugation on $\C_{\exp}$;
it has been posed in various forms by Kirby~\cite{Kirby2009a}, Macintyre,
Onshuus~\cite{Kirby2012}, Zilber and others if these involutions
exist, and possibly if we can make the pseudoexponentiation increasing
on the fixed field of the involution, which is automatically an ordered
real closed field. It is known that Zilber fields have many automorphisms,
but not if there is an involution, and on the other hand, the only
automorphism of $\C_{\exp}$ that we know of is complex conjugation.

Moreover, the involution of $\C_{\exp}$ plays a rather important
role, as we can use it to give a simple definition of $\exp$: it
is the unique exponential function commuting with complex conjugation
and continuous with respect to the induced topology such that $\exp(1)=e$,
$\ker(\exp)=2\pi i\Z$ and $\exp(i\frac{\pi}{2})=i$. If we could
find an involution of $\B_{E}$ whose fixed field, as a pure field,
is isomorphic to $\R$, such that $E$ is continuous in the induced
topology, and the three above equations are satisfied, then we would
have a proof of Zilber's conjecture and Schanuel's conjecture at once.

The author announced in \cite{Mantova2011} the positive answer to
the first of the above questions, at least for Zilber fields of cardinality
up to the continuum, including $\B$:
\begin{thm}
\label{thm:main}If $K_{E}$ is a Zilber field of cardinality $|K|\leq2^{\aleph_{0}}$,
there exists an involution $\sigma$ on $K_{E}$, i.e., there is a
field automorphism $\sigma:K\to K$ of order two such that $\sigma\circ E=E\circ\sigma$.

Moreover, if $K_{E}=\B$, there is a $\sigma$ such that $K^{\sigma}\cong\R$
and $\ker(E)=2\pi i\Z$.
\end{thm}
Here we describe the complete proof of this statement. In order to
deduce the theorem, we actually take the opposite route: we start
from a given automorphism $\sigma$ of the field $K$, and we construct
a suitable function $E$. We prove the following:
\begin{thm}
\label{thm:from-invo}There is a function $E:\C\to\C^{\times}$ such
that $\C_{E}$ is a Zilber field, $E(\overline{z})=\overline{E(z)}$
for all $z\in\C$, and $E\left(2\pi i(p/q)\right)=e^{2\pi i(p/q)}$
for all $p\in\Z$, $q\in\Z^{\times}$.

More generally, let $K$ be an algebraically closed field of characteristic
zero, $\sigma:K\to K$ an automorphism of order two, and $\omega$
a transcendental number in the fixed field $K^{\sigma}$.

If $K$ has infinite transcendence degree and the order topology of
$K^{\sigma}$ is second-countable, there is a function $E:K\to K^{\times}$
such that $K_{E}$ is a Zilber field, $\sigma\circ E=E\circ\sigma$
and $\ker(E)=i\omega\Z$.
\end{thm}
Note that if the order topology of $K^{\sigma}$ is second-countable,
the cardinality of $K$ cannot be greater than $2^{\aleph_{0}}$.

As a corollary of the above arguments, we obtain that there are actually
several non-isomorphic involutions on Zilber fields, as we can choose
to fill or omit different Dedekind cuts over $\Q$ of each $t_{j}$
(see \prettyref{thm:invo-dense}).

On the other hand, our construction has a strong limitation: it produces
a function $E$ that is not continuous with respect to the topology
induced by $\sigma$, even if we restrict to $K^{\sigma}$. The topology
induced by $\sigma$ \emph{is} related to $E$, but in a rather different
way from $\C_{\exp}$ with complex conjugation: we get dense sets
of solutions on rotund varieties of depth $0$ rather than isolated
solutions (see \prettyref{sec:involuntary}).

\subsection*{Structure of the paper}

We start by recalling the definition of Zilber field, together with
the notation used in the paper, in \prettyref{sec:defaxioms}.

The proof of \prettyref{thm:from-invo} starts in \prettyref{sec:cons}
with the description of the inductive construction that leads to $K_{E}$.
The proof that the construction actually yields the desired exponential
field is given in Sections \ref{sec:facts} to \ref{sec:proof}. \prettyref{sec:involuntary}
describes the limits of the current construction.

Sections \ref{sec:facts} and \ref{sec:ccp} describe a few facts
about Zilber fields, most of them already known and appearing in literature
in some form. \prettyref{sec:roots} describes a few properties of
the ``system of roots'' of rotund varieties introduced in \prettyref{sec:cons}.

Sections \ref{sec:g-res} to \ref{sec:roots-g-res} define the ``$\G$-restrictions''
of rotund varieties, which are essentially the semi-algebraic varieties
underlying each rotund variety; these sections contain a few geometric
results which are essential to the proof of the main theorem.

\prettyref{sec:proof} concludes the proof of Theorems \ref{thm:from-invo}
and \ref{thm:main}, putting together all the information obtained
in the previous sections.

In \prettyref{sec:further} there are a few variations on the construction
that let us obtain new exponential fields with involutions and some
nice extra properties.

\subsection*{Acknowledgements}

The author would like to thank Jonathan Kirby, who inspired this work,
and suggested many corrections, \foreignlanguage{italian}{Alessandro
Berarducci}, who followed the author throughout the project, Angus
Macintyre for his review, \foreignlanguage{italian}{Umberto Zannier}
for the encouraging comments, \foreignlanguage{italian}{Maurizio Monge}
for the useful discussions, and an anonymous referee for pointing
out problems and inconsistencies in an earlier version of the paper.

This project has been part of the author's PhD work at the \foreignlanguage{italian}{Scuola
Normale Superiore} of Pisa, and it has been partially supported by
the PRIN-MIUR 2009 ``\foreignlanguage{italian}{O-mi\-ni\-ma\-li\-tà,
teo\-ria de\-gli in\-sie\-mi, me\-to\-di e mo\-del\-li non\-stan\-dard
e ap\-pli\-ca\-zio\-ni}'', the EC's Seventh Framework Programme
{[}FP7/2007-2013{]} under grant agreement no.~238381, and the FIRB
2010 ``\foreignlanguage{italian}{Nuo\-vi svi\-lup\-pi nel\-la
Teo\-ria dei Mo\-del\-li del\-l'es\-po\-nen\-zia\-zio\-ne}''.

%% file: DefAxioms.tex
\section{\label{sec:defaxioms}Definitions and axioms}

We recall briefly the basic concepts needed to work with Zilber fields,
and the list of their axioms. We also fix some notation and some naming
conventions, mostly borrowing from~\cite{Kirby2009a}. A more succinct
account of the axioms, but with a different notation, can be found
in \cite{Marker2006}.

Before entering the details, we fix the context we work in to the
class of partial $E$-fields.
\begin{defn}
A \emph{partial exponential field}, or \emph{partial $E$-field} for
short, is a two-sorted structure 
\[
\langle\langle K;0,1,+,\cdot\rangle;\langle D;0,+,(q\cdot)_{q\in\Q}\rangle;i:D\to K,E:D\to K\rangle
\]
satisfying the axiom ``(E-par)'':

\begin{table}[H]
\begin{tabular}{ll}
\multirow{4}{*}{(E-par)} & $\langle D;0,+,(q\cdot)_{q\in\Q}\rangle$ is a $\Q$-vector space;\tabularnewline
 & $\langle K;0,1,+,\cdot\rangle$ is a field of characteristic $0$;\tabularnewline
 & $i:D\to K$ is an injective homomorphism from $\langle D,+\rangle$
to $\langle K,+\rangle$;\tabularnewline
 & $E$ is a homomorphism from $\langle D,+\rangle$ to $\langle K^{\times},\cdot\rangle$.\tabularnewline
\end{tabular}
\end{table}

A \emph{global $E$-field}, or just \emph{$E$-field}, is a two sorted
structure as above where the axiom (E) is satisfied:
\begin{enumerate}
\item [(E)]the axiom (E-par) holds, and $i$ is surjective.
\end{enumerate}
\end{defn}
We denote (partial) $E$-fields with the notation $K_{E}$. Unless
otherwise stated, \emph{we identify the vector space $D$ with its
image in $K$}, leaving implicit the extra sort, except for the few
situations where this generate ambiguities.

\addtocontents{toc}{\SkipTocEntry}

\subsection{Definitions}

Let us take some notation from Diophantine geometry: we denote by
$\G_{a}$ the additive group and by $\G_{m}$ the multiplicative group.
In other words, we have $\G_{a}(K)=(K,+)$ and $\G_{m}(K)=(K^{\times},\cdot)$.
We call $\G$ the product $\G_{a}\times\G_{m}$, and we denote its
group law by $\oplus$. The group $\G$ is a natural environment where
to look at points of the form $(z,E(z))\in\G_{a}\times\G_{m}=\G$.
Note that the group law is such that $(z,E(z))\oplus(w,E(w))$ is
just $(z+w,E(z+w))$. With a little abuse of notation, we shall often
use $\oplus$ to denote the group law on $\G^{n}$ as well.

The group $\G$ is naturally a $\Z$-module as an abelian group:
\begin{eqnarray*}
(\cdot):\Z\times\G & \to & \G\\
m\cdot(z,w) & \mapsto & (m\cdot z,w^{m}).
\end{eqnarray*}

The action can be naturally generalised to matrices with integer coefficients.
Given a matrix $M\in\Mc_{k,n}(\Z)$ of the form $M=(m_{i,j})_{1\leq i\leq k,1\leq j\leq n}$,
the explicit action can be written as
\begin{eqnarray*}
(\cdot):\Mc_{k,n}(\Z)\times\G^{n} & \to & \G^{k}\\
M\cdot(z_{j},w_{j})_{j} & \mapsto & \left(\sum_{j=1}^{n}m_{i,j}z_{j},\prod_{j=1}^{n}w_{j}^{m_{i,j}}\right)_{i}.
\end{eqnarray*}

In several situations, we will abbreviate the multiplication by the
scalar $m$, i.e., ``$(m\mathrm{Id})\cdot(z_{j},w_{j})_{j}$'',
with the expression $m\cdot(z_{j},w_{j})$.

For the sake of readability, we introduce a special notation for the
elements of $\G^{n}$ and for the function $E$.
\begin{notation}
For any positive integer number $n$, and for any two given vectors
$\overline{z}=(z_{1},\dots,z_{n})\in K^{n}$ and $\overline{w}=(w_{1},\dots,w_{n})\in(K^{\times})^{n}$,
we denote by $\intv{\overline{z}}{\overline{w}}$ the `interleaved'
vector $((z_{1},w_{1}),\dots,(z_{n},w_{n}))\in\G^{n}(K)$.

If $\overline{z}=(z_{1},\dots,z_{n})\in\dom(E)^{n}$, we write $E(\overline{z})$
to denote $(E(z_{1}),\dots,E(z_{n}))\in(K^{\times})^{n}$. If $\overline{z}\in K^{n}$,
but not all the components of $\overline{z}$ lie in $\dom(E)$, we
still write $E(\overline{z})$ to denote the set of the images of
the components of $\overline{z}$ in $\dom(E)$ (in this case, a pure
set with less than $n$ elements).
\end{notation}
With this notation, the following equation holds
\[
M\cdot\zEz{\overline{z}}=\zEz{M\cdot\overline{z}},
\]
where $M\cdot\overline{z}$ denotes the usual action of matrices on
vector spaces.
\begin{defn}
An irreducible subvariety $V$ of $\G^{n}$, for some positive integer
$n$, is \emph{rotund }if for all $M\in\Mc_{n,n}(\Z)$ the following
inequality holds:
\[
\dim M\cdot V\geq\rank M.
\]

\end{defn}
In the original conventions of~\cite{Zilber2005}, a rotund variety
is called ``ex-normal'', or just ``normal'', reusing a terminology
in common with other amalgamation constructions. However, we preferred
the convention of~\cite{Kirby2009a} not to risk confusion with the
term ``normal'' from algebraic geometry.

Note that we do not specify the base field over which $V$ is irreducible;
it can happen that $V$ is irreducible, but not absolutely irreducible,
as it can split into finitely many subvarieties when enlarging the
field of definition. The absolutely irreducible components of a rotund
variety, i.e., the ones irreducible over the algebraic closure, are
still rotund.

Let us call $\pi_{a}$ and $\pi_{m}$ the projections of $\G$ on
its two factors $\G_{a}$ and $\G_{m}$ resp.
\begin{defn}
An irreducible subvariety $V$ of $\G^{n}$, for some integer $n$,
is \emph{absolutely free} if for all non-zero matrices $M\in\Mc_{n,1}(\Z)$
the following holds:
\[
\dim\pi_{a}(M\cdot V)=\dim\pi_{m}(M\cdot V)=1.
\]

\end{defn}
In Zilber's paper \cite{Zilber2005} there is a relative notion of
``freeness''. We do not need it here. As above, if $V$ is absolutely
free, all of its absolutely irreducible components will be absolutely
free.

Absolutely free rotund varieties are to be thought as systems of equations
in one iteration of $E$ which are compatible with the Schanuel Property
(see below) and are not overdetermined. Indeed, here are their solutions.
\begin{defn}
If $V$ is an absolutely free rotund variety $V\subset\G^{n}$, a
vector $\overline{z}\in K^{n}$ is a \emph{solution} of $V$ if $\zEz{\overline{z}}\in V$.
\end{defn}
Note that when $\overline{z}$ is a solution of $V$, then for all
$q\in\N^{\times}$ the vector $\frac{1}{q}\overline{z}$ is a solution
to $W$ for some absolutely irreducible variety with the property
that $q\cdot W=V$. The number of the varieties $W$ such that $q\cdot W=V$
is always bounded by $q^{n}$, but in many useful cases there is only
one of them. It is becoming common to call Kummer-generic the varieties
with this property.
\begin{defn}
An absolutely irreducible algebraic variety $V\subset\G^{n}$ is \emph{Kummer-generic}
if for all $q\in\N^{\times}$ there is only one absolutely irreducible
variety $W$ such that $q\cdot W=V$.
\end{defn}
If $V$ is not absolutely free, then it is easy to see that it cannot
be Kummer-generic. On the other hand, if $V$ is absolutely free,
it is not far from being Kummer-generic in a precise sense.
\begin{fact}
\label{fact:thumbtack}If $V$ is absolutely free and absolutely irreducible,
there exists an integer $q\in\N^{\times}$ such that each absolutely
irreducible variety $W$ satisfying $q\cdot W=V$ is Kummer-generic.
\end{fact}
This is an instance of the Thumbtack Lemma \cite[Thm.\ 2]{Zilber2006},
and it actually holds when replacing $\G_{m}^{n}$ with any semiabelian
variety and even in some more general contexts \cite{Bays2011a}.
\begin{notation}
Given a subset $\overline{c}\subset K$, we write $\overline{c}E(\overline{c})$
to denote the pure set of the elements of $\overline{c}$ together
with the image of $\overline{c}\cap\dom(E)$ through $E$.

We write $V(\overline{c})$ to denote a variety $V$ and a finite
tuple of parameters $\overline{c}$ such that $V$ is defined over
$\overline{c}E(\overline{c})$. In this case, we say that $V$ is
\emph{$E$-defined over $\overline{c}$}.
\end{notation}
The notion of $E$-defined is particularly handy when we deal with
predimensions (see \prettyref{sec:facts}). Note, however, that it
depends on the function $E$; whenever we want to refer to a different
function, say $E'$, we will write ``$E'$-defined''. We will use
the notation $V(\overline{c})$ only when it is clear from the context
which function $E$ we are considering.
\begin{defn}
A solution $\overline{z}\in K^{n}$ of $V$ is a \emph{generic solution
of $V(\overline{c})$ in $K_{E}$} when the corresponding point $\zEz{\overline{z}}\in V$
is generic over $\overline{c}E(\overline{c})$ in the algebraic sense,
i.e., when
\[
\td_{\overline{c}E(\overline{c})}\zEz{\overline{z}}=\dim V.
\]

A set of generic solutions $\mathcal{S}$ of $V(\overline{c})$ is
\emph{algebraically independent in $K_{E}$} if for any finite subset
$\{\overline{z}_{1},\dots,\overline{z}_{k}\}\subset\mathcal{S}$ the
following holds:
\[
\td_{\overline{c}E(\overline{c})}(\zEz{\overline{z}_{1}},\dots,\zEz{\overline{z}_{k}})=k\cdot\dim V.
\]

\end{defn}
Note that genericity depends on the parameters we choose to define
$V$ and on $E$, hence the reference to $K_{E}$ and to $V(\overline{c})$
rather than $V$ in the definition of generic solution. When it is
clear from the context, we shall drop the reference to $K_{E}$.

A special role is given to the varieties $V\subset\G^{n}$ such that
$\dim V=n$.
\begin{defn}
The \emph{depth} of a variety $V\subset\G^{n}$ is $\delta(V):=\dim V-n$.
\end{defn}
Hence a variety with $\dim V=n$ has ``depth $0$''. For all rotund
varieties, $\delta(V)\geq0$. The generic solutions of an absolutely
free rotund variety $V(\overline{c})$ of depth $0$ are said to be
exponential-algebraic over $\overline{c}$, and vice versa, an exponential-algebraic
element over $\overline{c}$ is a component of a generic solution
of some $V(\overline{c})$ of depth $0$.
\begin{defn}
An absolutely free rotund variety $V\subset\G^{n}$ is \emph{simple}
if for all $M\in\Mc_{n,n}(\Z)$ with $0<\rank M<n$ the following
strict inequality holds:
\[
\dim M\cdot V>\rank M.
\]

\end{defn}
The partial $E$-field extension generated by a generic solution of
a simple variety has no proper exponential-algebraic subextension,
hence simple varieties correspond to simple extensions in Hrushovski's
amalgamation terminology.
\begin{defn}
An absolutely free rotund variety $V\subset\G^{n}$ is \emph{perfectly
rotund} if it is simple and $\delta(V)=0$.
\end{defn}
\addtocontents{toc}{\SkipTocEntry}

\subsection{Zilber's axioms}

Now that we have given all the relevant definitions, we can list the
axioms defining Zilber fields. Let $K_{E}$ be our structure. We split
the axioms into three groups, depending on their meaning for $\C_{\exp}$.

\subsubsection{Trivial properties of $\C_{\exp}$}
\begin{enumerate}
\item [(ACF\textsubscript{0})] $K$ is an algebraic closed field of characteristic
$0$.
\item [(E)] $E$ is a homomorphism $E:(K,+)\to(K^{\times},\cdot)$.
\item [(LOG)] $E$ is surjective (every element has a logarithm).
\item [(STD)] the kernel is a cyclic group, i.e., $\ker E=\omega\Z$ for
some $\omega\in K^{\times}$.
\end{enumerate}
Note that the axiom (E) describes a \emph{global} function, so the
extra sort $D$ of partial $E$-fields becomes redundant, as the function
$i:D\to K$ would be a bijection.

\subsubsection{Axioms conjecturally true on $\C_{\exp}$}
\begin{enumerate}
\item [(SP)] \emph{Schanuel's Property}: for every finite tuple $\overline{z}=(z_{1},\dots,z_{n})\in K$
such that $z_{1},\dots,z_{n}$ are $\Q$-linearly independent,
\[
\td_{\Q}\zEz{\overline{z}}=\td_{\Q}(z_{1},\dots,z_{n},E(z_{1}),\dots,E(z_{n}))\geq n.
\]

\item [(SEC)] \emph{Strong Exponential-algebraic Closure}: for every absolutely
free rotund variety $V\subset\G^{n}$ over $K$, and every finite
tuple $\overline{c}\in K^{<\omega}$ such that $V$ is defined over
$\overline{c}$, $V(\overline{c})$ has a generic solution $\overline{z}\in K^{n}$.
\end{enumerate}

\subsubsection{A non-trivial property of $\C_{\exp}$ \cite[Lemma 5.12]{Zilber2005}}
\begin{enumerate}
\item [(CCP)] \emph{Countable Closure Property}: for every absolutely free
rotund variety $V\subset\G^{n}$ over $K$ of depth $0$, and every
finite tuple $\overline{c}\in K^{<\omega}$ such that $V$ is defined
over $\overline{c}$, the set of the generic solutions of $V(\overline{c})$
is at most countable.
\end{enumerate}
If $K_{E}$ satisfies the above seven axioms we say that $K_{E}$
is a Zilber field. By Zilber's theorem \cite{Zilber2005} and corrections
\cite{Bays2012,Bays2013}, if $K_{E}$ is uncountable, its isomorphism
type is completely determined by its cardinality, and in general it
is determined by its exponential transcendence degree. The one of
cardinality $2^{\aleph_{0}}$ is called $\B$, $\B_{E}$, $\B_{\mathrm{ex}}$
or $\F_{\mathrm{ex}}$, and it is sometimes called \emph{the} Zilber
field, as it is the one conjecturally isomorphic to $\C_{\exp}$.

Some authors add the axiom ``$K_{E}$ has infinite exponential transcendence
degree'' as well, so that also the countable model is determined
just by the cardinality, but we do not need this extra assumption
here.

%% file: Construction.tex
\section{\label{sec:cons}The construction}

As anticipated in the introduction, the proof of \prettyref{thm:from-invo}
is obtained with an explicit construction. Our starting data are an
algebraically closed field $K$ of characteristic $0$, an involution
$\sigma$ of $K$ and some transcendental $\omega$ in the fixed field
$K^{\sigma}$. Our result is a function $E$ such that $K_{E}$ is
a Zilber field, $\sigma\circ E=E\circ\sigma$ and $\ker(E)=i\omega\Z$.
We describe here the procedure, and the proof that the resulting $K_{E}$
has the desired properties will be given throughout the rest of the
paper.

The idea is to define $E$ inductively by back-and-forth: at one step,
we define the function on a new element of $K$, at another we define
it so that a new element of $K^{\times}$ appears in the image, and
in the meanwhile we add solutions to each rotund variety (in a careful
way).

\subsection{Commuting with $\sigma$}

In order to have $\sigma\circ E=E\circ\sigma$ at each stage, we use
a rather easy observation. Let us copy the usual notation for $\C$
and $\R$ to our case. We denote by $R$ the fixed field $K^{\sigma}$,
which is a real closed field; if $i$ is a square root of $-1$, we
have $K=R(i)$.

We define:
\begin{enumerate}
\item the \emph{real part} of $z\in K$ as $\Re(z):=\frac{z+\sigma(z)}{2}$;
\item the \emph{imaginary part} of $z\in K$ as $\Im(z):=\frac{z-\sigma(z)}{2i}$;
\item the \emph{modulus} of $z\in K^{\times}$ as $|z|:=\sqrt{z\cdot\sigma(z)}$;
\item the \emph{phase} of $z\in K^{\times}$ as $\Theta(z):=\frac{z}{|z|}$.
\end{enumerate}
The former two are the additive decomposition of a number of $K$
over $R$, the latter are its multiplicative decomposition. The image
of the function $\Theta$ is the `unit circle', and it will be denoted
by $\Sb^{1}=\{z\in K^{\times}\,:\,|z|=1\}$.

Using this description, it is quite easy to see the following.
\begin{prop}
\label{prop:commutes} Let $K_{E}$ be a partial $E$-field, and $\sigma:K\to K$
be a field automorphism of order two. Then $\sigma\circ E=E\circ\sigma$
if and only if the following three conditions are satisfied:
\begin{enumerate}
\item $\sigma(\dom(E))=\dom(E)$;
\item $E(R)\subset R_{>0}$;
\item $E(iR)\subset\Sb^{1}$;
\end{enumerate}
\end{prop}
\begin{proof}
If $\sigma\circ E=E\circ\sigma$, it is clear that $\sigma(\dom(E))=\dom(E)$.
For all $x\in R$ we have $\sigma(E(x))=E(x)$, which implies $E(x)\in R$,
and since $E(x)=E\left(\frac{x}{2}\right)^{2}$, it is actually $E(x)\in R_{>0}$;
moreover, for all $y\in R$ we have $\sigma(E(iy))=E(-iy)=E(iy)^{-1}$,
i.e., $|E(iy)|=\sqrt{\sigma(E(iy))E(iy)}=1$.

On the other hand, suppose that the three conditions are satisfied.
We have then that for all $x,y\in R$, $\sigma(E(x))=E(x)$ and $\sigma(E(iy))=E(iy)^{-1}$.
Since any $z\in K$ can be written uniquely as $x+iy$ with $x,y\in R$,
and if $E$ is defined over $z$, by the first condition it is also
defined over $x=\Re(z)$ and $y=\Im(z)$, we have
\[
\sigma(E(z))=\sigma(E(x))\sigma(E(iy))=E(x)E(iy)^{-1}=E(x-iy)=E(\sigma(z)).\qedhere
\]

\end{proof}
Hence, in our back-and-forth construction, every time we define the
function on a new $z$ linearly independent from the current domain,
it shall be sufficient to make sure that either $z\in R$ or $z\in iR$,
and that we define the new values according to the above restrictions.

\subsection{An induction}

To begin, let $(\zeta_{q})_{q\in\N^{\times}}$ be a coherent system
of roots of unity (i.e., such that $\zeta_{pq}^{p}=\zeta_{q}$). We
start with the function $E_{0}$ defined by $E_{0}(i\frac{p}{q}\omega):=\zeta_{q}^{p}$,
with domain $\dom(E_{0})=i\omega\Q$, and we proceed by transfinite
induction, with two slightly different procedures depending on whether
$K$ is countable or not.

Let us enumerate the relevant objects:
\begin{enumerate}
\item let $\{\alpha_{j}\}_{j<|K|}$ be an enumeration of $R\cup iR$;
\item let $\{\beta_{j}\}_{j<|K|}$ be an enumeration of $R_{>0}\cup\Sb^{1}$;
\item let $\{V_{j}\}_{j<|K|}$ be an enumeration of all the simple Kummer-generic
varieties over $K$, where each variety appears only once if $K$
is uncountable, and countably many times if $K$ is countable.
\end{enumerate}
The construction itself is based on a certain number of basic `operation'
that we use to extend the partial exponential functions.

In each of the operations, we take a partial $E$-field $K_{E}$ such
that $\sigma\circ E=E\circ\sigma$, and we extend $E$ to a new function
$E'$ with the same properties. We denote by $D$ the domain of $E$,
and by $F$ the field generated by $D$ and $E(D)$. We always assume
that $F$ has infinite transcendence degree over $K$.
\begin{description}
\item [{\textsc{domain}}] We start with a given $\alpha\in R\cup iR$.
If $\alpha\in D$, we define $E':=E$, otherwise we do the following.\\
If $\alpha\in R$, we choose a $\beta\in R_{>0}\setminus\acl(F\cup\{\alpha\})$
and we let $(\beta^{1/q})$ be its \emph{positive} roots; if $\alpha\in iR$,
we choose a $\beta\in\Sb^{1}\setminus\acl(F\cup\{\alpha\})$ and we
let $(\beta^{1/q})$ be some coherent system of roots.\\
We define $E'(z+\frac{p}{q}\alpha):=E(z)\cdot\beta^{p/q}$ for all
$z\in D$ and $p\in\Z,q\in\N^{\times}$.
\item [{\textsc{image}}] We start with a given $\beta\in R_{>0}\cup\Sb^{1}$.
If $\beta\in E(D)$, we define $E':=E$, otherwise we do the following.\\
If $\beta\in R_{>0}$, we choose an $\alpha\in R\setminus\acl(F\cup\{\beta\})$
and we let $(\beta^{1/q})$ be the \emph{positive} roots of $\beta$;
if $\beta\in\Sb^{1}$, we choose an $\alpha\in iR\setminus\acl(F\cup\{\beta\})$
and we let $\beta^{1/q}$ be some coherent system of roots.\\
We define $E'(z+\frac{p}{q}\alpha):=E(z)\cdot\beta^{p/q}$ for all
$z\in D$ and $p\in\Z,q\in\N^{\times}$.
\end{description}
These operations are used to guarantee that the final function $E$
is surjective and defined everywhere. We also want to verify (SEC),
so we use a special operation to add generic solutions to rotund varieties.
In order to preserve the property $\sigma\circ E=E\circ\sigma$ while
maintaining (SP), we use points such that not only their transcendence
degree is as large as possible, but also the transcendence degree
of their real and imaginary parts is.

For reasons that will become apparent later, in the uncountable case
we also want the solutions to be dense in the order topology. In order
to do so, we first prepare an operation where we add one solution
to a given rotund variety on a specified open set.
\begin{description}
\item [{\textsc{sol}}] We start with an absolutely irreducible variety
$V(\overline{c})\subset\G^{n}$, where $\overline{c}$ is a subset
of $K$ closed under $\sigma$, and a subset $U\subset V$ open in
the order topology.\\
We choose a point $((\alpha_{1}+i\gamma_{1},\beta_{1}\cdot\delta_{1}),\dots,(\alpha_{n}+i\gamma_{n},\beta_{n}\cdot\delta_{n}))\in U$
with the following properties:
\begin{eqnarray*}
 & \alpha_{j},\gamma_{j}\in R,\,\beta_{j}\in R_{>0},\,\delta_{j}\in\Sb^{1}(K)\mbox{\quad for }1\leq j\leq n\\
 & \td_{F(\overline{c}E(\overline{c}))}(\alpha_{1},\gamma_{1},\beta_{1},\delta_{1},\dots)=2\dim V.
\end{eqnarray*}
 We fix the positive roots $\beta_{j}^{1/q}$ of $\beta_{j}$ and
an arbitrary system of roots $\delta_{j}^{1/q}$, and we define
\[
E'\left(z+\frac{p_{1}}{q_{1}}\alpha_{1}+i\frac{p'_{1}}{q'_{1}}\gamma_{1}+\dots\right):=E(z)\cdot\beta_{1}^{p_{1}/q_{1}}\delta_{1}^{p'_{1}/q'_{1}}\cdot\dots.
\]

\end{description}
Now, it would be sufficient to repeat the above operation on a basis
of open sets of $V$, so that the solutions of $V$ would be dense.
However, again for reasons that will be clear later, we do this in
a slightly different way.
\begin{defn}
Let $V$ be an absolutely irreducible variety. We call the \emph{system
$\Rc(V)$ of the roots of $V$} the set of all the absolutely irreducible
varieties $W$ such that for some $q\in\Z^{\times}$ we have $q\cdot W=V$.
\end{defn}
Instead of just taking a variety $V$ and adding a dense set of solutions
to it, we take all the varieties in the system $\Rc(V)$ and we add
solutions to all of them at once. We pack this procedure into a single
operation. Here we assume that $F$ has transcendence degree at least
$\aleph_{1}$ over $K$, and that $K^{\sigma}$ is separable.
\begin{description}
\item [{\textsc{roots}}] We start with an absolutely irreducible variety
$V(\overline{c})\subset\G^{n}$, where $\overline{c}$ is a finite
subset of $K$ closed under $\sigma$.\\
Consider an enumeration $(W_{m}(\overline{d}_{m}),U_{m})_{m<\omega}$
of all the pairs composed by a variety $W_{m}$ of $\Rc(V)$ and an
open subset $U_{m}\subset W_{m}$ chosen among a fixed countable basis
of the order topology on $W_{m}$, where $\overline{d}_{m}$ is a
finite subset of $\acl(\overline{c}E(\overline{c}))$ over which $W$
is $E$-defined. We produce inductively a sequence of partial exponential
functions $E_{m}$, starting with $E_{0}:=E$.\\
Let us suppose that $E_{m-1}$ has been defined. If $W_{m}(\overline{d}_{m})$
has a dense set of really algebraically independent solutions in $K_{E_{m-1}}$,
we define $E_{m}:=E_{m-1}$; otherwise we apply \textsc{sol} to $(W_{m}(\overline{d}_{m}),U_{m})$
over the $E$-field $K_{E_{m-1}}$. The resulting exponential function
will be $E_{m}$.\\
Finally, we define $E':=\bigcup_{m\in\N}E_{m}$.
\end{description}
These are the \emph{finite} operations. There is only one non-finite
operation. Let $(E_{k})_{k<j}$ be an increasing sequence of partial
exponential functions.
\begin{description}
\item [{\textsc{limit}}] We define $E':=\bigcup_{k<j}E_{k}$.
\end{description}
The actual construction follows.
\begin{enumerate}
\item If $j=k+1$ is a successor ordinal, we do the following:

\begin{enumerate}
\item \label{enu:rdom}We apply \textsc{domain} to $\alpha_{k}$ and $E_{k}$
to obtain $E_{k}^{(1)}$.
\item \label{enu:rimg}We apply \textsc{image} to $\beta_{k}$ and $E_{k}^{(1)}$
to obtain $E_{k}^{(2)}$.
\item \label{enu:rsol-dom}We take the variety $V_{k}$ and a finite set
$\overline{c}$ of parameters closed under $\sigma$ such that $V_{k}$
is $E_{k}^{(2)}$-defined over $\overline{c}$.\\
We apply \textsc{domain} to each element of $\overline{c}$ to obtain
$E_{k}^{(3)}$.
\item \label{enu:rsol}At last, if $R$ is uncountable, we apply \textsc{roots}
to $V_{k}(\overline{c})$ and $E_{k}^{(3)}$; otherwise, we apply
\textsc{sol} to $V_{k}(\overline{c})$, with $V_{k}$ itself as open
subset, and $E_{k}^{(3)}$. The function resulting from this operation
is $E_{j+1}$.
\end{enumerate}
\item \label{enu:rlim}If $j$ is a limit ordinal we apply \textsc{limit}
to the sequence $(E_{k})_{k<j}$ to obtain the function $E_{j}$.
\end{enumerate}
We shall prove in \prettyref{sec:proof} that $K_{E_{|K|}}$ is indeed
a Zilber field. Most properties are easy to check; the first actual
obstacle is proving that \textsc{sol} does not falsify (SP) (a consequence
of \prettyref{thm:rotund}).

However, the biggest difficulty comes from verifying that $K_{E_{|K|}}$
satisfies (CCP). In order to check (CCP), we look at how many times
a rotund variety $X(\overline{c})$ of depth $0$ receives new generic
solutions as an effect of the above operations. After some reductions
on the shape of $X$ described in Sections\ \ref{sec:facts} and
\ref{sec:ccp}, we may see that it happens essentially only during
an application of \textsc{sol}, and that the variety $V$ to which
we apply the operation must be strongly related to $X$ (\prettyref{lem:new-sol-alg-map}).

Because of \prettyref{thm:rotund}, the relationship between $V$
and $X$ will let us move the solutions from $V$ to $X$ and \emph{back}
(\prettyref{prop:dense-solutions-in-roots})\emph{. }Roughly speaking,
the use of dense sets of solutions will guarantee that after countably
many applications of \textsc{roots}, the solutions of $X$ can be
always moved back to solutions of the potential $V$'s, making the
further applications of \textsc{roots} that could contribute to $X$
void. But since we stop adding new solutions to $X$ after countably
many steps, $X(\overline{c})$ has only countably many generic solutions,
proving (CCP).

%% file: Involuntary.tex
\section{\label{sec:involuntary}An (involuntary) restriction on $\sigma$}

Before entering the details of the proof of \prettyref{thm:from-invo},
we wish to comment the fact that the resulting exponential function
behaves rather badly with respect to the topology induced by $\sigma$,
and quite differently from the classical case $\C_{\exp}$.

Indeed, we can easily describe a restricted class of involutions containing
the ones obtained with our method. Let $(K_{E},\sigma)$ be a structure
where $K_{E}$ is a Zilber field and $\sigma$ is an involution of
$K_{E}$. Consider the following axioms:
\begin{enumerate}
\item [(ZIL)] $K_{E}$ is a Zilber field;
\item [(INV)] $\sigma$ is an automorphism of order two;
\item [(DEN)] for every absolutely free rotund variety $V\subset\G^{n}$
over $K$, every finite tuple $\overline{c}\in K^{<\omega}$ such
that $V$ is defined over $\overline{c}$, and every subset $U\subset V$
open w.r.t.\ the order topology of $K^{\sigma}$, there is a generic
solution $\overline{z}\in K^{n}$ of $V(\overline{c})$ such that
$\zEz{\overline{z}}\in U$.
\end{enumerate}
The last axioms means that for each absolutely free rotund variety
$V(\overline{c})$ the set of its generic solutions is dense in the
order topology. In some sense, we are saying that $E$ is not only
random with respect to the field structure, but also with respect
to the order structure. It can be expressed by a first order formula
with just a slight modification of axiom (SEC).

Our proof shows explicitly the existence of many models in the above
class.
\begin{thm}
\label{thm:invo-dense}For each cardinal $0<\kappa\leq2^{\aleph_{0}}$,
there are $2^{\aleph_{0}+\kappa}$ non-isomorphic models of the axioms
(ZIL), (INV), (DEN) of dimension $\kappa$. There is at least one
model of dimension $0$.
\end{thm}
It is a direct consequence of our construction that there is one such
model built over $\C$ and complex conjugation. Starting with one
such model, we may choose a maximal set $\{t_{j}\}_{j\in J}$ of exponential-algebraic
elements contained in the fixed field $\R$; for each $J_{0}\subset J$,
the subfield $\cl(\{t_{j}\}_{j\in J})$ is itself a model of the above
axioms, and since it fills the Dedekind cuts of $t_{j}$ for $j\in J_{0}$,
but omits the ones for $j\notin J_{0}$, it is non-isomorphic to the
others obtained with different subsets of $J$. This shows the existence
of the desired number of non-isomorphic models.

However, our method does not say anything about models where the axiom
(DEN) is not true. In particular, even if $\C_{\exp}$ is a Zilber
field, and we take the classical complex conjugation as $\sigma$,
the structure $(\C_{\exp},\sigma)$ cannot be a result of our construction.

Moreover, (CCP) and (DEN) imply that the topology must be separable,
hence the cardinality of $K$ is forced to be at most $2^{\aleph_{0}}$.

%% file: Facts.tex
\section{\label{sec:facts}Some facts about Zilber fields}

In order to prove that the resulting $E$-field $K_{E_{|K|}}$ of
\prettyref{sec:cons} is a Zilber field, we need to check that each
of the axioms listed in \prettyref{sec:defaxioms} is true. Using
some known properties and tools of Zilber fields, we can reduce a
bit the complexity of this problem, especially for (SEC) and (CCP).

Here we recall two equivalent, but simpler formulations of the axioms
(SEC) and (CCP), and some properties of the predimension of Zilber
fields that helps in verifying (SP).

\subsection{Equivalent formulations}

We can verify that the axioms (SEC) and (CCP) can be checked on simple
and perfectly rotund Kummer-generic varieties only, rather than all
rotund varieties. This is the reason why it is sufficient to use simple
Kummer-generic varieties in the step~\prettyref{enu:rsol} of our
construction. Moreover, at least for (SEC) the parameters defining
the varieties are actually not relevant.
\begin{fact}
\label{fact:sec-is-minus} Let $K_{E}$ be a partial $E$-field. (SEC)
holds on $K_{E}$ if and only if the following holds:
\begin{enumerate}
\item [\emph{(SEC\textsubscript{1})}] for any simple Kummer-generic variety
$V$ defined over \emph{some} finite $\overline{c}$, $V(\overline{c})$
has an infinite set of algebraically independent solutions.
\end{enumerate}
\end{fact}

\begin{fact}
\label{fact:ccp-is-minus} Let $K_{E}$ be a partial $E$-field. (CCP)
holds on $K_{E}$ if and only if the following holds:
\begin{enumerate}
\item [\emph{(CCP\textsubscript{1})}] for any perfectly rotund Kummer-generic
variety $V$, and for any finite $\overline{c}$ such that $V$ is
defined over $\overline{c}$, $V(\overline{c})$ has at most countably
many generic solutions.
\end{enumerate}
\end{fact}
Hence, it will be sufficient to verify (SEC\textsubscript{1}) and
(CCP\textsubscript{1}) instead of their full versions. To prove the
above statements, we use the following fact.
\begin{prop}
\label{prop:indep-up-to-finite}Let $V$ be an absolutely free rotund
variety. Let $\overline{c}$ and $\overline{d}$ be two finite tuples
such that $V$ is $E$-defined both over $\overline{c}$ and $\overline{d}$.

If $\Sc$ is a set of algebraically independent solutions of $V(\overline{c})$,
then $\Sc$, up to removing a finite set, is an algebraically independent
set of solutions of $V(\overline{d})$.\end{prop}
\begin{proof}
For each finite subset $\Sc'\subset\Sc$, let $\Delta(\Sc)$ be the
following quantity:
\[
\Delta(\Sc'):=\td_{\overline{c}E(\overline{c})}(\Sc')-\td_{\overline{d}E(\overline{d})}(\Sc')=\dim V\cdot|\Sc'|-\td_{\overline{d}E(\overline{d})}(\Sc').
\]

$\Delta$ measures how far is $\Sc'$ from being algebraically independent
over $\overline{d}E(\overline{d})$. First of all, we claim that $\Delta(\Sc')$
is bounded from above. Indeed,
\[
\td_{\overline{d}E(\overline{d})}(\Sc')\geq\td_{\overline{c}\overline{d}E(\overline{c}\overline{d})}(\Sc')\geq\td_{\overline{c}E(\overline{c})}(\Sc')-\td_{\overline{c}E(\overline{c})}(\overline{d}E(\overline{d})),
\]
 and after shuffling the terms, $\Delta(\Sc')\leq\td_{\overline{c}E(\overline{c})}(\overline{d}E(\overline{d}))$.
Moreover, $\Delta$ is clearly an increasing function.

Now let $\Sc_{0}$ be a finite set such that $\Delta(\Sc_{0})$ is
maximum. We claim that $\Sc\setminus\Sc_{0}$ is algebraically independent.
Let us consider a finite subset $\Sc'\subset\Sc\setminus\Sc_{0}$;
since the function $\Delta$ is increasing, we have
\begin{eqnarray*}
 & \dim V\cdot(|\Sc'|+|\Sc_{0}|)-\td_{\overline{d}E(\overline{d})}(\Sc'\cup\Sc_{0})=\Delta(\Sc'\cup\Sc_{0})=\\
 & =\Delta(\Sc_{0})=\dim V\cdot|\Sc_{0}|-\td_{\overline{d}E(\overline{d})}(\Sc_{0}).
\end{eqnarray*}

This implies that
\[
\dim V\cdot|\Sc'|\leq\td_{\overline{d}E(\overline{d}),\Sc_{0}}(\Sc')\leq\td_{\overline{d}E(\overline{d})}(\Sc')=\dim V\cdot|\Sc'|,
\]
 as desired.
\end{proof}

\begin{proof}[Proof of Facts \ref{fact:sec-is-minus} and \ref{fact:ccp-is-minus}]

One direction is trivial for both statements, so let us see the other
direction.

Let $V$ be an absolutely free rotund variety $E$-defined over $\overline{c}$.
By \prettyref{fact:thumbtack}, if we add some elements of $\acl(\overline{c}E(\overline{c}))$
to $\overline{c}$, for some appropriate integer $q\in\N^{\times}$
we find that $V$ is of the form $q\cdot(W_{1}\cup\dots\cup W_{m})$
with $W_{1},\dots,W_{m}$ Kummer-generic, absolutely free rotund varieties
$E$-defined over $\overline{c}$. If we verify that (SEC) and (CCP)
hold when specialised in each $W_{j}$, then (SEC) and (CCP) are also
true when specialised in $V$. Hence, we may assume that $V$ is Kummer-generic.

We proceed by induction on $n=\dim(V)$. If $V$ is simple, and (SEC\textsubscript{1})
is true, then $V$ has infinitely many algebraically independent solutions
over some set of parameters $\overline{d}$, and by \prettyref{prop:indep-up-to-finite},
after removing a finite number of them, they are also algebraically
independent over $\overline{c}$. If $V$ is perfectly rotund, and
(CCP\textsubscript{1}) is true, there are at most countably many
generic solutions.

Now, let us suppose that $V$ is not simple, and that we have proved
the conclusions for all the varieties of dimension smaller than $V$.
Let $z_{1},w_{1,}\dots,z_{n},w_{n}$ be the coordinate functions of
$V$.

Let $M$ be an integer matrix such that $0<k=\rank M<n$ and $\dim M\cdot V=\rank M$.
Using a suitable square invertible matrix, and discarding the coordinates
that become zero, we may assume that $M$ is the projection of $V$
over the coordinates $z_{1},w_{1},\dots,z_{k},w_{k}$.

Let $N$ be any matrix in $\Mc_{h,n-k}(\Z)$ of rank $h$, with $h\leq n-k$.
By rotundity, we have
\[
\td_{\overline{c}E(\overline{c})}\left(\left(\begin{array}{c|c}
\mathrm{Id}_{k} & 0\\
\hline 0 & N
\end{array}\right)\cdot(z_{1},w_{1},\dots,z_{n},w_{n})\right)\geq k+h.
\]

But since $\dim M\cdot V=k$, this means
\[
\td_{\overline{c}E(\overline{c}),z_{1},\dots,z_{k},w_{1},\dots,w_{k}}(N\cdot(z_{k+1},w_{k+1},\dots,z_{n},w_{n}))\geq h=\rank N.
\]

In other words, whenever we specialise the first $2k$ coordinates
to a generic solution of $M\cdot V$, the remaining coordinates describe
a \emph{rotund} variety of dimension smaller than $\dim V$.

The projection $M\cdot V$ is a rotund variety of depth $0$ and of
dimension smaller than $\dim(V)$. If $\tilde{z}=(\tilde{z}_{1},\dots,\tilde{z}_{k})$
is a solution of $M\cdot V$, and we specialise the variables $z_{1},w_{1},\dots,z_{k},w_{k}$
to $\tilde{z}_{1},E(\tilde{z}_{1}),\dots,\tilde{z}_{k},E(\tilde{z}_{k})$
and project onto the last $2(n-k)$ coordinates, we obtain a new variety
$\tilde{V}(\overline{c},\tilde{z})$; by the above argument, $\tilde{V}$
is a \emph{rotund} variety of dimension smaller than $\dim V$.

If (SEC\textsubscript{1}) is true, by inductive hypothesis $M\cdot V(\overline{c})$
has at least countably many algebraically independent solutions; let
us call them $\{\tilde{z}_{j}\}_{j\in\N}$. Again by inductive hypothesis,
the variety $\tilde{V}_{0}(\overline{c},\tilde{z}_{0})$ obtained
from $\tilde{z}_{0}$ as above has a generic solution, say $\overline{y}_{0}$,
and by construction, the concatenation $\tilde{z}_{0}\overline{y}_{0}$
is a generic solution of $V(\overline{c})$. Moreover, the variety
$\tilde{V}_{1}(\overline{c},\tilde{z}_{1},\tilde{z}_{0},\overline{y}_{0})$
has a generic solution $\overline{y}_{1}$, hence $\tilde{z}_{1}\overline{y}_{1}$
is a generic solution of $V(\overline{c})$ algebraically independent
from $\tilde{z}_{0}\overline{y}_{0}$. We can proceed further by induction,
taking a generic solution of $\tilde{V}_{n}(\overline{c},\tilde{z}_{n},\tilde{z}_{0}\overline{y}_{0},\dots,\tilde{z}_{n-1}\overline{y}_{n-1})$
as $\overline{y}_{n}$ for each $n\in\N$. The resulting $\{\tilde{z}_{n}\overline{y}_{n}\}_{n\in\N}$
form an algebraically independent set of solutions for $V(\overline{c})$,
proving our original inductive step.

On the other hand, if (CCP\textsubscript{1}) is true, the inductive
hypothesis tells us that $M\cdot V(\overline{c})$ has at most countably
many generic solutions $\tilde{z}$. Moreover, for each such $\tilde{z}$,
the generic solutions of $\tilde{V}(\overline{c},\tilde{z})$ are
at most countably many. Since each generic solution $\overline{z}$
of $V(\overline{c})$ can be decomposed as $\tilde{z}\overline{y}$,
with $\tilde{z}$ generic solution of $M\cdot V(\overline{c})$ and
$\overline{y}$ generic solution of $\tilde{V}(\overline{c},\tilde{z})$,
the variety $V(\overline{c})$ has at most countably many solutions.

Therefore, (SEC\textsubscript{1}) implies (SEC) and (CCP\textsubscript{1})
implies (CCP).
\end{proof}

\subsection{Predimension}

The axiom (SP) can be interpreted as stating that a certain quantity
is always positive, and in the context of Hrushovski's amalgamation,
this quantity works as a predimension on $K_{E}$. The predimension
is particularly useful when dealing with exponential fields, even
when (SP) is not true, and it is a crucial tool in Zilber's proof
of categoricity of Zilber fields. In our case, we actively use its
machinery to verify that (SP) holds on $K_{E_{|K|}}$.

Let $K_{E}$ be a partial $E$-field. Given a set $X$ contained in
the domain of $E$, we denote by $E(X)$ the image of the elements
of $X$ through $E$.
\begin{defn}
Let $X\subset\dom(E)$ be a finite set. The \emph{predimension} of
$X$ is the quantity
\[
\delta(X)=\td_{\Q}(X\cup E(X))-\ld_{\Q}(X).
\]

If $X$ is a finite subset of $K$, and $Y$ an arbitrary subset of
$K$ such that $E$ is defined on $X$ and $Y$, we define
\[
\delta(X/Y):=\td_{\Q}(X\cup E(X)/Y\cup E(Y))-\ld_{\Q}(X/Y).
\]

\end{defn}
In these definitions, $\td_{\Q}(X)$ stands for the transcendence
degree of $X$ over the base field $\Q$, while $\ld_{\Q}(X)$ means
the $\Q$-linear dimension of $X$. Instead, $\td_{\Q}(X/Y)$ is the
transcendence degree of $X$ over $\Q(Y)$, and $\ld_{\Q}(X/Y)$ is
the linear dimension of $X$ over the $\Q$-linear span of $Y$. In
order to reduce the size of the formulas, we will often write $\td_{Y}(X)$
to abbreviate $\td_{\Q}(X/Y)$. 

With this notation, (SP) is equivalent to $\delta(X)\geq0$ for all
$X$. Note that this is the same $\delta$ used to denote the depth
of a variety $V$: if $\overline{z}$ is a generic solution of $V(\overline{c})$,
then $\delta(\overline{z}/\overline{c})=\delta(V)$ (we are using
the fact that $V$ is $E$-defined over $\overline{c}$, rather than
just ``defined over $\overline{c}$'', otherwise we would have to
assume that $\overline{c}\leq K$).
\begin{rem}
The definition of predimension is slightly different from~\cite{Zilber2005},
as we calculate it only on the domain of $E$. We are following again~\cite{Kirby2009a},
as it greatly simplifies the discussion.\end{rem}
\begin{defn}
Let $K_{E}\subset K'_{E'}$ be two partial $E$-fields. Let $\delta$
be the predimension function on $K'_{E'}$.

We say that $K_{E}\leq K'_{E'}$, or that $K_{E}$ is \emph{strongly
embedded} in $K'_{E'}$, if for each finite $X\subset\dom(E')$ we
have $\delta(X/\dom(E))\geq0$.

If $X\subset\dom(E)$ we say that $X$ is \emph{strong} in $K_{E}$,
$X\leq K_{E}$, if $K_{E_{\restriction\span_{\Q}(X)}}\leq K_{E}$.
\end{defn}
Some facts about predimensions, and in particular their relationship
with rotund varieties, are useful tools in proofs. See \cite{Kirby2009a}
for reference.
\begin{fact}
The axiom (SP) holds on $K_{E}$ if and only if $\{0\}\leq K_{E}$.

\label{fact:strong-keeps-sc}If $K_{E}\leq K'_{E'}\leq K''_{E''}$,
then $K_{E}\leq K''_{E''}$. In particular, if $K_{E}$ satisfies
(SP) and $K_{E}\leq K'_{E'}$, then $K'_{E'}$ satisfies (SP).
\end{fact}

\begin{fact}
\label{fact:asp-strong}Let $K_{E}$ a partial $E$-field satisfying
(SP). Then for all finite $X\subset\dom(E)$ there is a finite $Y\subset\dom(E)$
such that $XY\leq K_{E}$.
\end{fact}
For the last statement, it is sufficient that $K_{E}$ satisfies an
``almost Schanuel's Property'' stating that there is a $k\in\N$
such that $\delta(Z)\geq-k$ for all $Z$. This implies that $\delta(XY)$
attains a minimum as $Y$ varies, proving the statement (see again
\cite{Kirby2009a}).
\begin{fact}
\label{fact:rotund-is-strong}Let $V$ be an absolutely free rotund
variety over $K$, and let $((z_{1},w_{1}),\dots,(z_{n},w_{n}))$
be a generic point of $V$ over $K$ into some field extension $K'$.

If we extend linearly $E$ to a function $E'$ by defining $E'(\frac{p}{q}z_{i})=w_{i}^{p/q}$,
then the new function is well defined and $K_{E}\leq K'_{E'}$. Moreover,
$\ker(E)=\ker(E')$.\end{fact}

%% file: CCP.tex
\section{\label{sec:ccp}Preserving CCP}

The most difficult axiom to verify on the $E$-field $K_{E_{|K|}}$
of \prettyref{sec:cons} is (CCP), even in its weaker form (CCP\textsubscript{1}).
In this section we show that (CCP) is preserved by the finite operations,
i.e., if (CCP) holds at some stage of our construction, it holds after
the application of one of the finite operations, and in particular
it holds also for the successive stage.

More generally, we claim that if $K_{E}$ satisfies (CCP), and we
extend linearly $E$ to $E'$ by defining $E'$ on some finite, or
even countable, set of elements that are $\Q$-linearly independent
over $\dom(E)$, then also $K_{E'}$ satisfies (CCP). This kind of
fact is not a novelty, but it has not been stated in literature, and
since it is quite important for our constructions, we describe it
here in full details.

We prove it in two steps. First, we show that there is an equivalent
formulation of (CCP) which is even simpler than (CCP\textsubscript{1}),
under the assumption that (SP) holds; then, using the simplified formulation,
we prove that (CCP) is preserved if we extend the function $E$ on
a not too large set.
\begin{prop}
\label{prop:pres-ccp-field}Let $K_{E}$ be a partial $E$-field satisfying
(SP). Then axiom (CCP) is equivalent to
\begin{enumerate}
\item [\emph{(CCP\textsubscript{2})}] for any perfectly rotund variety
$V$ $E$-defined over $\dom(E)$, and for any finite tuple $\overline{c}\subset\dom(E)$
such that $V$ is $E$-defined over $\overline{c}$, there are at
most countably many generic solutions of $V(\overline{c})$.
\end{enumerate}
\end{prop}
\begin{proof}
The left-to-right direction is clear. Let us prove the other direction.

Let us suppose that $K_{E}$ satisfies (CCP\textsubscript{2}). Let
$X(\overline{c})\subset\G^{n}$ be a perfectly rotund variety. Without
loss of generality, we may assume that $\overline{c}$ is of the form
$\overline{c}_{0}\overline{c}_{1}$, with $\overline{c}_{0}\subset\dom(E)$
and $\overline{c}_{1}\cap\dom(E)=\emptyset$.

Let $\overline{d}$ be a finite subset of $\dom(E)$ such that
\begin{enumerate}
\item $\overline{c}_{0}\overline{d}\leq K_{E}$;
\item $\td_{\overline{c}_{0}\overline{d}E(\overline{c}_{0}\overline{d})}(\overline{c}_{1})=\td_{\dom(E),\im(E)}(\overline{c}_{1})$.
\end{enumerate}
It exists by \prettyref{fact:asp-strong}: just add elements to $\overline{c}_{0}$
until the second equality is satisfied, then add new elements so that
the first requirement is satisfied.

Now, let us take a generic solution $\overline{z}$ of $X(\overline{c})$.
There is an invertible matrix $M$ with coefficients in $\Z$ such
that $M\cdot\overline{z}$ is of the form $\overline{z}_{0}\overline{z}_{1}$,
with $\overline{z}_{0}\subset\span_{\Q}(\overline{c}_{0}\overline{d})$
and $\overline{z}_{1}$ $\Q$-linearly independent from $\span_{\Q}(\overline{c}_{0}\overline{d})$.

By hypothesis, $\overline{c}_{0}\overline{d}\leq K_{E}$, so for any
matrix $N$ with integer coefficients:
\[
\td_{\overline{c}_{0}\overline{d}E(\overline{c}_{0}\overline{d})}\zEz{N\cdot\overline{z}_{1}}\geq\ld_{\Q}(N\cdot\overline{z}_{1}/\overline{c}_{0}\overline{d})=\rank N.
\]

Moreover, since $\overline{z}_{0}\overline{z}_{1}$ is a generic point
of $M\cdot X(\overline{c})$, which is a perfectly rotund variety
defined over $\overline{c}$, we have
\begin{eqnarray*}
 & \td_{\overline{c}\overline{z}_{0}E(\overline{c}\overline{z}_{0})}\zEz{\overline{z}_{1}}=\\
 & =\td_{\overline{c}E(\overline{c})}\zEz{\overline{z}_{0}\overline{z}_{1}}-\td_{\overline{c}E(\overline{c})}\zEz{\overline{z}_{0}}\leq\\
 & \leq|\overline{z}_{0}\overline{z}_{1}|-|\overline{z}_{0}|=|\overline{z}_{1}|,
\end{eqnarray*}
 where by $|\overline{x}|$ we mean the length of the vector $\overline{x}$.
We claim that the above inequality must be an equality. Indeed, since
$\overline{z}_{0}\subset\span_{\Q}(\overline{c}_{0}\overline{d})$,
we have 
\[
\td_{\overline{c}_{1}\overline{c}_{0}\overline{d}E(\overline{c}_{0}\overline{d})}\zEz{\overline{z}_{1}}\leq\td_{\overline{c}\overline{z}_{0}E(\overline{c}\overline{z}_{0})}\zEz{\overline{z}_{1}}\leq|\overline{z}_{1}|.
\]

But since $\td_{\overline{c}_{0}\overline{d}\overline{z}_{1}E(\overline{c}_{0}\overline{d}\overline{z}_{1})}(\overline{c}_{1})=\td_{\overline{c}_{0}\overline{d}E(\overline{c}_{0}\overline{d})}(\overline{c}_{1})$,
we also have
\[
|\overline{z}_{1}|\leq\td_{\overline{c}_{0}\overline{d}E(\overline{c}_{0}\overline{d})}\zEz{\overline{z}_{1}}=\td_{\overline{c}_{1}\overline{c}_{0}\overline{d}E(\overline{c}_{0}\overline{d})}\zEz{\overline{z}_{1}}\leq|\overline{z}_{1}|,
\]

implying that the first inequality was actually an equality.

By perfect rotundity, the equality holds only if either $|\overline{z}_{0}|=0$
or $|\overline{z}_{1}|=0$. In the latter case, $\overline{z}\subset\span_{\Q}(\overline{c}_{0}\overline{d})$,
so that there are at most countably many such $\overline{z}$'s. In
the former case, $\overline{z}$ is a generic solution of a variety
$E$-defined over $\dom(E)$, so they are at most countably many solutions
by (CCP\textsubscript{2}). This implies (CCP\textsubscript{1}),
and by \prettyref{fact:ccp-is-minus}, it implies (CCP) as well, as
desired.
\end{proof}
Note that this also explicitly shows what one expects: (CCP) really
depends only on $E$ and not on the base field $K$.
\begin{prop}
\label{prop:pres-ccp}Let $K_{E}\subset K'_{E'}$ be two partial $E$-field
satisfying (SP).

If $\dom(E')=\span_{\Q}(\dom(E)\cup B)$ for some finite or countable
$B\subset K'$, then $K_{E}$ satisfies (CCP) if and only if $K'_{E'}$
does.\end{prop}
\begin{proof}
Clearly, if $K'_{E'}$ satisfies (CCP), then $K_{E}$ does too. For
the other direction, let us suppose that $K_{E}$ satisfies (CCP).

By \prettyref{prop:pres-ccp-field}, it is sufficient to prove that
$K'_{E'}$ satisfies (CCP\textsubscript{2}). Moreover, we may assume
that $K=K'$.

Let $X(\overline{c})$ be a perfectly rotund variety $E'$-defined
over $\overline{c}$, with $\overline{c}\subset\dom(E')$. 

Let us take a generic solution of $X(\overline{c})$ in $K_{E'}$;
by assumptions, it can be written uniquely as $\overline{z}+M\cdot\overline{b}$,
with $\overline{z}\subset\dom(E)$, $\overline{b}$ a finite subset
of $B$ and $M$ a matrix with coefficients in $\Q$. We claim that
given $M$ and $\overline{b}$, there are at most countably many $\overline{z}$'s
such that $\overline{z}+M\cdot\overline{b}$ is a generic solution
of $X(\overline{c})$ in $K_{E'}$.

Let $\overline{d}$ be a finite subset of $\dom(E')$ such that $\overline{b}\overline{c}\overline{d}\leq K'_{E'}$.

There is an invertible matrix $N$ with coefficients in $\Z$ such
that $N\cdot\overline{z}$ is the concatenation $\overline{z}_{0}\overline{z}_{1}$,
with $\overline{z}_{0}\subset\span_{\Q}(\overline{b}\overline{c}\overline{d})$
and $\overline{z}_{1}$ $\Q$-linearly independent from $\span_{\Q}(\overline{b}\overline{c}\overline{d})$.
Thus $(N\cdot\overline{z}+N\cdot M\cdot\overline{b})$ is a generic
solution of $N\cdot X(\overline{c})$ in $K_{E'}$, which is again
a perfectly rotund variety $E'$-defined over $\overline{c}$. For
the sake of notation, let $\overline{b}_{0}\overline{b}_{1}$ be the
splitting of $N\cdot M\cdot\overline{b}$ corresponding to $\overline{z}_{0}\overline{z}_{1}$.

By the hypothesis, $\overline{b}\overline{c}\overline{d}\leq K'_{E'}$,
so for any matrix $P$ with integer coefficients
\[
\td_{\overline{b}\overline{c}\overline{d}E'(\overline{b}\overline{c}\overline{d})}\zEz{P\cdot\overline{z}_{1}}\geq\rank P.
\]

Moreover, since $\overline{z}_{0}\overline{z}_{1}+\overline{b}_{0}\overline{b}_{1}$
is a generic solution of $N\cdot X(\overline{c})$ in $K_{E'}$, we
have
\[
\td_{\overline{c}(\overline{z}_{0}+\overline{b}_{0})E'(\overline{c}(\overline{z}_{0}+\overline{b}_{0}))}\zEz{\overline{z}_{1}+\overline{b}_{1}}\leq|\overline{z}_{1}|.
\]

As before, we claim that the above inequality is an equality. Since
$\overline{z}_{0}\subset\span_{\Q}(\overline{b}\overline{c}\overline{d})$,
we have
\[
\td_{\overline{b}\overline{c}\overline{d}E'(\overline{b}\overline{c}\overline{d})}\zEz{\overline{z}_{1}}\leq\td_{\overline{c}(\overline{z}_{0}+\overline{b}_{0})E'(\overline{c}(\overline{z}_{0}+\overline{b}_{0}))}\zEz{\overline{z}_{1}+\overline{b}_{1}}\leq|\overline{z}_{1}|.
\]

And in particular, 
\[
|\overline{z}_{1}|\leq\td_{\overline{b}\overline{c}\overline{d}E'(\overline{b}\overline{c}\overline{d})}\zEz{\overline{z}_{1}}\leq|\overline{z}_{1}|.
\]

By perfect rotundity, we must have either $|\overline{z}_{0}|=0$
or $|\overline{z}_{1}|=0$. In the latter case, we have that $\overline{z}\subset\span_{\Q}(\overline{b}\overline{c}\overline{d})$,
so that there are at most countably many such $\overline{z}$'s. In
the former case, $\overline{z}$ is a generic solution of a perfectly
rotund variety $E'$-defined over $\overline{b}\overline{c}\overline{d}$,
which can be seen as a perfectly rotund variety $E$-defined over
$\overline{b}\overline{c}\overline{d}E'(\overline{b}\overline{c}\overline{d})$.
But there are only countably many such varieties, each one with countably
many generic solutions in $K_{E}$; therefore, there are at most countably
many such $\overline{z}$'s.

But there are at most countably many matrices $M$ and vectors $\overline{b}$,
and for each of them, at most countably many $\overline{z}$'s; therefore,
there are at most countably many generic solutions of $X(\overline{c})$
in $K'_{E'}$, i.e., (CCP) holds on $K'_{E'}$.
\end{proof}
This clearly implies that if at some step of our construction we have
a $K_{E_{j}}$ satisfying (CCP), then $K_{E_{j+1}}$ satisfies (CCP)
as well, since $\dom(E_{j+1})$ has finite or countable dimension
over $\dom(E_{j})$. In particular, (CCP) is preserved by any of the
finite operations. 

However, more work is needed to deal with the \textsc{limit} operation.

%% file: Roots.tex
\section{\label{sec:roots}Solutions and roots}

In order to control (CCP) at the \textsc{limit} operation, we need
to carefully count how many solutions of perfectly rotund varieties
appear during the construction.

Here we concentrate on the way in which generic solutions can be transferred
from one variety to another by multiplication by matrices and translations.
The facts proved here will be particularly useful in counting how
many times the operation \textsc{roots} can produce new generic solutions
of a fixed perfectly rotund variety.

First of all, we recall that the varieties in $\Rc(V)$ are defined
over $\acl(\overline{c}E(\overline{c}))$.
\begin{prop}
If $V$ is $E$-defined over $\overline{c}$, then all the varieties
$W\in\Rc(V)$ are $E$-defined over $\acl(\overline{c}E(\overline{c}))$.\end{prop}
\begin{proof}
Clearly, if $q\cdot W=V$, then $W$ is defined over $\acl(\overline{c}E(\overline{c}))$.
\end{proof}
From now on, let $K_{E}$ be a partial $E$-field.
\begin{defn}
\label{def:comp-solved}A system $\Rc(V)$, with $\overline{c}$ $E$-defining
$V$, is \emph{completely solved} if for all $W\in\Rc(V)$ there is
an infinite set of solutions of $W$ algebraically independent over
$\acl(\overline{c}E(\overline{c}))$.
\end{defn}
By \prettyref{prop:indep-up-to-finite}, this definition does not
depend on the choice of $\overline{c}$.

It is easy to see that (SEC) is equivalent to saying that all the
systems $\Rc(V)$ are completely solved. Indeed, the operation \textsc{roots}
applied to $V$ makes sure that $\Rc(V)$ is completely solved.

The following facts describe the relationships that can occur between
the solutions of system of roots for different varieties.
\begin{prop}
Let $V\in\G^{n}$ be an absolutely irreducible rotund variety, and
let $M\in\Mc_{k,n}(\Z)$ be an integer matrix.

If $W\in\Rc(V)$, then $M\cdot W\in\Rc(M\cdot V)$.\end{prop}
\begin{proof}
If $W\in\Rc(V)$, then there is a $q\in\N^{\times}$ such that $q\cdot W=V$.
Multiplying by $M$ we obtain $M\cdot(q\cdot W)=M\cdot V$. However,
$M$ commutes with $q\cdot\mathrm{Id}$, hence $q\cdot(M\cdot W)=M\cdot V$,
i.e., $M\cdot W\in\Rc(M\cdot V)$.\end{proof}
\begin{cor}
Let $V\in\G^{n}$ be an absolutely irreducible rotund variety, and
let $M\in\Mc_{n,n}(\Z)$ be a square integer matrix of maximum rank
such that $M\cdot V$ is Kummer-generic.

If $W\in\Rc(M\cdot V)$, then there is a $W'\in\Rc(V)$ such that
$M\cdot W'=W$. Moreover, $V$ is Kummer-generic.\end{cor}
\begin{proof}
By definition of $\Rc(M\cdot V)$, there is an integer $q\in\N^{\times}$
such that $W$ is an absolutely irreducible variety such that $q\cdot W=M\cdot V$;
but since $M\cdot V$ is Kummer-generic, there is only one such $W$.
Let $W'$ be any absolutely irreducible variety such that $q\cdot W'=V$.
The variety $M\cdot W'$ is such that $q\cdot(M\cdot W')=M\cdot(q\cdot W')=M\cdot V$;
by uniqueness, we must have $M\cdot W'=W$.

Let $\tilde{M}$ be the integer matrix such that $\tilde{M}\cdot M=|\det M|\cdot\mathrm{Id}$,
and let $q\in\N^{\times}$ be any non-zero natural number. Let $W_{0}\in\Rc(V)$
be a variety such that $q\cdot W_{0}=V$. Now, let $W_{1}\in\Rc(M\cdot W_{0})$
be the variety such that $|\det M|\cdot W_{1}=M\cdot W_{0}$. By the
previous argument, there is a $W_{2}\in\Rc(W_{0})$ such that $M\cdot W_{2}=W_{1}$;
this implies that $\tilde{M}\cdot W_{1}=\tilde{M}\cdot M\cdot W_{2}=|\det M|\cdot W_{2}=W_{0}$.
This implies that there is only one $W_{0}$ such that $q\cdot W_{0}=V$,
showing that $V$ is Kummer-generic.\end{proof}
\begin{prop}
Let $V\subset\G^{n}$ be an absolutely irreducible rotund variety
and let $\overline{z}\in\dom(E)^{n}$.

If $W\in\Rc(V\oplus\zEz{\overline{z}})$, then there is an integer
$q\in\N^{\times}$ and a variety $W'\in\Rc(V)$ such that $W'\oplus\zEz{\frac{1}{q}\overline{z}}=W$.\end{prop}
\begin{proof}
By definition, there is $q\in\N^{\times}$ such that $q\cdot W=V\oplus\zEz{\overline{z}}$.
Let $W':=W\oplus\zEz{-\frac{1}{q}\overline{z}}$. The second part
of the thesis is satisfied.

Moreover, 
\[
q\cdot W'=q\cdot\left(W\oplus\zEz{-\frac{1}{q}\overline{z}}\right)=q\cdot W\oplus\zEz{-\overline{z}}=V.
\]

Hence, $W'\in\Rc(V)$.
\end{proof}
Now we can say something about how generic solutions of one system
of roots move to solutions of another system. The following result
is not useful yet in the proof of the main theorem, but it is the
kind of statement we are looking for. It is also useful for other
constructions, as shown in \prettyref{sec:further}.
\begin{prop}
\label{prop:solutions-in-roots} Let $V\subset\G^{n}$ be a rotund
variety. Let $M\in\Mc_{n,n}(\Z)$ be a square integer matrix of maximum
rank and $\overline{z}\in\dom(E)^{n}$.

If $M\cdot V$ is Kummer-generic, then the system $\Rc(V)$ is completely
solved if and only if $\Rc(M\cdot V\oplus\zEz{\overline{z}})$ is.\end{prop}
\begin{proof}
It is sufficient to verify the left-to-right direction of the implication.
Indeed, if $\tilde{M}$ is the integer matrix such that $\tilde{M}\cdot M=|\det M|\cdot\mathrm{Id}$,
and $W$ is the unique variety such that $|\det M|\cdot W=M\cdot V\oplus\zEz{\overline{z}}$,
we can also write $V=\tilde{M}\cdot W\oplus\zEz{-\frac{\tilde{M}}{|\det M|}\cdot\overline{z}}$.
Since $\Rc(W)\subset\Rc(M\cdot V\oplus\zEz{\overline{z}})$, the two
sides can be exchanged to reverse the argument. Hence, from now on
let us suppose that $\Rc(V)$ is completely solved.

Let $W\in\Rc(M\cdot V\oplus\zEz{\overline{z}})$. By the above propositions,
there are a $W'\in\Rc(V)$ and an integer $q\in\N^{\times}$ such
that $M\cdot W'\oplus\zEz{\frac{1}{q}\overline{z}}=W$.

Let $\overline{c}$ be a finite set of parameters $E$-defining $W'$
containing also $\overline{z}$. Clearly, $W$ is $E$-defined also
over $\overline{c}$, and if $\overline{x}$ is a generic solution
of $W'(\overline{c})$, then $M\cdot\overline{x}+\frac{1}{q}\overline{z}$
is a generic solution of $W(\overline{c})$.

Moreover, we claim that the map $P\mapsto M\cdot P\oplus\zEz{\frac{1}{q}\overline{z}}$,
for $P\in W'$, preserves the algebraic independence over $\overline{c}E(\overline{c})$.
As the translation by $\zEz{\frac{1}{q}\overline{z}}$ is an algebraic
invertible map defined over $\acl(\overline{c}E(\overline{c}))$,
it is sufficient to check this on the map $P\mapsto M\cdot P$.

However, since $M$ is invertible, there is an integer matrix $\tilde{M}$
such that $\tilde{M}\cdot M=|\det M|\cdot\mathrm{Id}$; in particular,
the composition $P\mapsto M\cdot P\mapsto\tilde{M}\cdot M\cdot P=|\det M|\cdot P$
is just the map $P\mapsto|\det M|\cdot P$. This map is algebraic
and finite-to-one, hence it preserves the algebraic independence;
in particular, $P\mapsto M\cdot P$ must preserve the algebraic independence.

This implies that an infinite set of algebraically independent solutions
of $W'(\overline{c})$ is mapped to an infinite set of algebraically
independent solutions of $W(\overline{c})$. In particular, if $\Rc(V)$
is completely solved, then $W$ contains an infinite set of algebraically
independent solutions. Since this holds for any $W$, $\Rc(M\cdot V\oplus\zEz{\overline{z}})$
is completely solved.
\end{proof}
However, something stronger is needed for our construction, as we
do not just add generic solutions to simple varieties, but we force
them to be dense, and moreover the solutions satisfy a transcendence
condition which is stronger than being generic.

%% file: GRestriction.tex
\section{\label{sec:g-res}$\G$-restriction of the scalars}

During our construction, we add solutions to simple varieties that
are not just generic, but something more. We put this into a definition.
From now on, we shall assume that $\sigma\circ E=E\circ\sigma$.
\begin{defn}
If $V$ is an absolutely free rotund variety $V(\overline{c})\subset\G^{n}$,
with $\overline{c}$ closed under $\sigma$, a generic solution $\overline{z}\in K^{n}$
is \emph{real generic} if
\[
\td_{\overline{c}E(\overline{c})}\zEz{\Re(\overline{z})\Im(\overline{z})}=2\dim V.
\]

A set of real generic solutions $\mathcal{S}$ of $V(\overline{c})$
is \emph{really algebraically independent} if for any finite subset
$\{\overline{z}_{1},\dots,\overline{z}_{k}\}\subset\mathcal{S}$ the
following holds:
\[
\begin{array}{c}
\td_{\overline{c}E(\overline{c})}(\zEz{\Re(\overline{z}_{1})\Im(\overline{z}_{1})},\dots,\zEz{\Re(\overline{z}_{k})\Im(\overline{z}_{k})})=\\
=2k\cdot\dim V.
\end{array}
\]

\end{defn}
In the operation \textsc{roots} we are adding real generic solutions
to simple varieties. Thus, we are actually adding generic solutions
to some varieties of twice the dimension; we give them a name.
\begin{defn}
We define the group $\G_{R}:=(R\times R_{>0})\times(iR\times\Sb^{1})\subset\G^{2}(K)$
and the \emph{realisation} map $r:\G\to\G_{R}$ as follows:
\[
r\,:\,(z,w)\mapsto(\Re(z),|w|)\times(i\Im(z),\Theta(w)).
\]
We extend $r$ as a map $\G^{n}\to\G^{2n}$ in the following way
\[
r:\intv{\overline{z}}{\overline{w}}\mapsto\intv{\Re(\overline{z})}{|\overline{w}|}\times\intv{\Im(\overline{z})}{\Theta(\overline{w})}\in(R\times R_{>0})^{n}\times(iR\times\Sb^{1})^{n}\subset\G^{n}(K).
\]
\end{defn}
\begin{rem}
It would have been more natural to define $r$ as the coordinate-wise
application $\G^{n}\to\G_{R}^{n}$; however, we will need to manipulate
matrices, and the above extension of $r$ to $\G^{n}$ lets us avoid
further special notations. Indeed, given a matrix $M\in\Mc_{k,2n}$,
its action on $r(V)$ can be easily divided in its action on the real
part plus its action on the imaginary part by just splitting the matrix
into two halves as 
\[
M=\left(\begin{array}{c|c}
N & P\end{array}\right).
\]
 The matrix $N$ acts on $(R\times R_{>0})^{n}$ and the matrix $P$
on $(iR\times\Sb^{1})^{n}$, and the result is put together using
$\oplus$.
\end{rem}
We apply the map $r$ to the points of rotund varieties.
\begin{defn}
Let $V$ be a subvariety of $\G^{n}$ for some $n$.
\begin{enumerate}
\item the \emph{realisation} of $V$ is the set $r(V):=\{r(\intv{\overline{z}}{\overline{w}})\in\G^{2n}\,:\,\intv{\overline{z}}{\overline{w}}\in V\}$;
\item the \emph{$\G$-restriction of the scalars} of $V$ is the Zariski
closure of $r(V)$ in $\G^{2n}$; it will be denoted by $\check{V}$.
\end{enumerate}
\end{defn}
In the operation \textsc{roots}, what actually happens is that when
we take a variety $W\in\Rc(V)$, we add generic solutions to $\check{W}$;
moreover, we choose the new solutions inside $r(W)$, and we make
sure that they are dense in $r(W)$.
\begin{rem}
The group $\G_{R}$ can be thought as a semi-algebraic group over
$R$ replacing $iR$ with $R$ and $\Sb^{1}$ with $\{(x,y)\in R^{2}\,:\, x^{2}+y^{2}=1\}$.
Then $r(V)$ can be seen as a semi-algebraic subvariety of $\G_{R}^{n}$.

The algebraic variety $\check{V}$ is similar to the classical Weil
restriction of the scalars. However, unlike the classical case, while
the points of $r(V)$ are in bijection with the points of $V$, the
set $\check{V}(R)$ is \emph{larger} than $r(V)$. This happens because
the factorisation $\rho\theta$ is unique only if we assume that $\rho>0$,
and this semi-algebraic condition is lost when passing to the Zariski
closure.
\end{rem}
The study of the $\G$-restriction of rotund varieties is quite crucial
for proving that our construction yields a Zilber field. It is essential
to prove that both (SP) and (CCP) holds in the final structure.

%% file: Rotundity.tex
\section{\label{sec:rotundity}Rotundity for $\G$-restrictions}

In order to prove that (SP) and (CCP) holds in the structure resulting
from \prettyref{sec:cons}, we need to state a structure theorem for
the $\G$-restriction $\check{V}$ of a rotund variety $V$. Our first
problem is to determine whether the $\G$-restriction of an absolutely
free rotund variety is itself absolutely free and rotund, so that
at least the construction is well-defined and (SP) is preserved.

We would like $\check{V}$ to be simple as well, in order to be able
to apply the results of the previous section to get (CCP), but even
if $V$ is simple, the corresponding $\check{V}$ is never simple,
as the following trivial equation implies:
\[
\dim\left(\begin{array}{c|c}
\mathrm{Id} & \mathrm{Id}\end{array}\right)\cdot\check{V}=\dim V=\rank\left(\begin{array}{c|c}
\mathrm{Id} & \mathrm{Id}\end{array}\right).
\]

We need to control how far $\check{V}$ is from being simple. It turns
out that the above example is essentially the only possible way in
which $\check{V}$ is not simple.
\begin{thm}
\label{thm:rotund}Let $V$ be an absolutely irreducible simple variety.
Then $\check{V}$ is an absolutely free, absolutely irreducible rotund
variety.

Moreover, if $\dim M\cdot\check{V}=\rank M$ for some integer matrix
in $\Mc_{p,2n}(\Z)$ of rank $p>0$, then $V$ is perfectly rotund,
and one of the following holds:
\begin{enumerate}
\item $\rank M=2n$;
\item $\rank M=n$, and $M$ is of the form
\[
M=\left(\begin{array}{c|c}
N & P\end{array}\right)
\]
where $N$, $P$ are two square matrices of maximum rank.
\end{enumerate}
\end{thm}
The proof requires several steps. Let us suppose that $\check{V}$
is $E$-defined over some $\overline{c}$ closed under $\sigma$.
\begin{prop}
\label{prop:abs-irred}If $V\subset\G^{n}$ is absolutely irreducible,
then $\check{V}$ is absolutely irreducible.\end{prop}
\begin{proof}
Let $V'$ be an absolutely irreducible variety such that $2\cdot V'=V$.

There is a map $V'\times(V')^{\sigma}\mapsto\G^{2n}$ described by
the following equation:
\[
\prod_{i=1}^{n}(z_{i},w_{i})\times\prod_{i=1}^{n}(z_{i}',w_{i}')\mapsto\prod_{i=1}^{n}\left(z_{i}+z_{i}',w_{i}w_{i}'\right)\times\prod_{i=1}^{n}\left(z_{i}-z_{i}',\frac{w_{i}}{w_{i}'}\right).
\]

It is clear that on the Zariski dense subset of $V'\times(V')^{\sigma}$
described by the points $P\times P^{\sigma}$, for $P\in V'$, the
image is exactly $r(V)$; taking the Zariski closure, we obtain that
this is a surjective map from $V'\times(V')^{\sigma}$ to $\check{V}$.

However, $V'\times(V')^{\sigma}$ is an absolutely irreducible variety,
as it is a product of two absolutely irreducible varieties; hence
$\check{V}$ is absolutely irreducible as well.\end{proof}
\begin{prop}
\label{prop:abs-free}If $V\subset\G^{n}$ is absolutely free, then
$\check{V}$ is absolutely free.\end{prop}
\begin{proof}
Let $x_{1},\dots,x_{n},y_{1},\dots,y_{n}$ be the additive coordinates
of $\check{V}$; by $x_{i}$ we mean the coordinates coming from the
real parts, and by $y_{i}$ the imaginary parts.

Let us suppose that for $m_{1},\dots,m_{n},p_{1},\dots,p_{n}\in\Z$
the image
\[
(m_{1}x_{1}+\dots+m_{n}x_{n}+p_{1}y_{1}+\dots+p_{n}y_{n})(\check{V})
\]
 is finite. In particular, it is true on the points of $r(V)$; this
implies that the images 
\[
(m_{1}x_{1}+\dots+m_{n}x_{n})(r(V)),\,(p_{1}y_{1}+\dots+p_{n}y_{n})(r(V))
\]
 are finite as well. However, this implies that the images
\[
(m_{1}z_{1}+\dots+m_{n}z_{n})(V),\,(p_{1}z_{1}+\dots+p_{n}z_{n})(V)
\]
 are not cofinite sets. By strong minimality of $K$, this implies
that $m_{1}z_{1}+\dots+m_{n}z_{n}$ and $p_{1}z_{1}+\dots+p_{n}z_{n}$
have both finite image, but by absolute freeness of $V$, this implies
$m_{1}=\dots=m_{n}=0$ and $p_{1}=\dots=p_{n}=0$.

The same argument applied to the multiplicative coordinates $\rho_{1},\dots,\rho_{n},\theta_{1},\dots,\theta_{n}$
yields the absolute freeness of $V$.
\end{proof}
In particular, when $V$ is an absolutely irreducible simple variety,
as it is in our construction, the variety $\check{V}$ is absolutely
free and absolutely irreducible. We still have to verify if it is
rotund, and how far it is from being simple.

From now on, let us suppose that $V$ is an absolutely irreducible
simple variety defined over some $\overline{c}$ closed under $\sigma$.

The main technical challenge in proving that $\check{V}$ is rotund
comes from taking the functions on $V$ as `complex-valued', or
`two-dimensional' functions (i.e., such that their image is the
complex plane minus a finite set), splitting them into components
making them `real-valued', or `one-dimensional', and then recombining
them into complex-valued functions again. The philosophy is that unless
the recombination happens in a special way, algebraic relations are
destroyed in the process. In order to prove this, we will introduce
a bit of ad hoc notation to deal with the mixed case where some functions
have been recombined into complex-valued functions, but some still
appear as real-valued.

The proof of the theorem will then start with finding a minimal matrix
$\check{M}$ such that the coordinate functions of $M\cdot r(V)$,
which can be mixed real and complex-valued, can be recovered from
the coordinate functions of $r(\check{M}\cdot V)$, which are the
real-valued components of the complex-valued functions of $\check{M}\cdot V$.
In particular, the additive coordinates of $M\cdot r(V)$ will be
recovered as $\Q$-linear combinations of the ones of $r(\check{M}\cdot V)$,
while the multiplicative coordinates of $M\cdot r(V)$ will be multiplicatively
dependent on the ones of $r(\check{M}\cdot V)$.

Once we have the right matrix $\check{M}$, it shall be sufficient
to transform the coordinate functions of $r(\check{M}\cdot V)$ to
coordinates of $M\cdot r(V)$. Keeping track of the dimension of $r(\check{M}\cdot V)$
when we drop the resulting extra functions that do not appear as coordinates
of $M\cdot r(V)$, we shall be able to deduce a lower bound on the
dimension of $M\cdot r(V)$, and in turn of $M\cdot\check{V}$. In
the case $\dim M\cdot\check{V}=\rank M$, all the inequalities will
be forced to be equalities, ultimately forcing $V$ to be perfectly
rotund in the first place, and $M$ to have the sought special form.

\subsection{Mixed functions\label{sub:mixed-functions}}

Let $S$ be an algebraically independent subset of the coordinate
functions of $M\cdot\check{V}$ for some matrix $M$. Taking the restrictions
to the subset $M\cdot r(V)\subset M\cdot\check{V}$, we can try to
estimate the actual size of $M\cdot\check{V}$ by studying how the
functions of $S$ behave on the points of $M\cdot r(V)$, more precisely
by looking at how large is their image.

Note that each function in $S$ is of the form $\overline{m}\cdot\overline{x}+\overline{q}\cdot\overline{y}$
or $\overline{\rho}^{\overline{m}}\overline{\theta}^{\overline{q}}$,
where $(\overline{m}|\overline{q})$ is a row of $M$ and $\overline{x},\overline{y},\overline{\rho},\overline{\theta}$
are the coordinate functions of $r(V)$. We are taking the coordinate
functions of $r(V)$ and not of its Zariski closure; in other words,
we are studying $r(V)$ as a semi-algebraic variety (although the
coordinates $\overline{y}$ and $\overline{\theta}$ are not numbers
in $R$). We introduce the following notation.
\begin{notation}
If $S$ is a set of coordinates as above, we denote by $r(S)$ the
subset of the coordinate functions of $r(V)$ containing all the $\overline{m}\cdot\overline{x}$,
$\overline{q}\cdot\overline{y}$, $\overline{\rho}^{\overline{m}}$,
$\overline{\theta}^{\overline{q}}$ such that $\overline{m}\cdot\overline{x}+\overline{q}\cdot\overline{y}$
and $\overline{\rho}^{\overline{m}}\overline{\theta}^{\overline{q}}$
are in $S$.
\end{notation}
We know that in general $\td_{\overline{c}}(S)\geq\td_{\overline{c}}(r(S))/2$,
but this is far from being enough for our purposes. We can produce
a better estimate by distinguishing, among the coordinate functions
in $S$, the `one-dimensional' functions from the `two-dimensional'
ones. The idea is that if one function is two-dimensional, then it
contributes with transcendence degree $1$ to $\td_{\overline{c}}S$,
but it needs two algebraically independent functions in $r(S)$ to
be calculated; on the other hand, a one-dimensional function only
needs one function of $r(S)$.

To make an example, suppose that $x_{1}+y_{1},x_{2}+y_{2},x_{3}+y_{3}$
are the coordinate functions of $\A^{3}(K)$, with $x_{j}$ and $y_{j}$
being their real and imaginary parts as in our current notation. Let
$S$ be $\{x_{1}+y_{1},x_{2}+y_{1},x_{3}+y_{1}\}$. The set $r(S)$
is equal to $\{x_{1},x_{2},x_{3},y_{1}\}$ and it is algebraically
independent; hence, our rough estimate on the transcendence degree
would just prove that $\td S\geq2$. However, it is clear that even
if we fix the values of $x_{1}+y_{1}$ and of $x_{2}+y_{1}$, the
remaining function $x_{3}+y_{1}$ is still able to vary in an infinite
set, so the transcendence degree is actually $\td S=3$. This can
be seen explicitly by noting that fixing $x_{1}+y_{1}$ actually fixes
the values of two functions in $r(S)$, but once that is done, fixing
the value of $x_{2}+y_{1}$ only fixes one function of $r(S)$, leaving
a third one free to vary. We would like to say that $x_{1}+y_{1}$
is `two-dimensional', but that $x_{2}+y_{1}$ is just `one-dimensional'.

The example shows the intuition we want to use, but also that a definition
of `one-dimensional' should depend on the order of the functions:
if we fix $x_{2}+y_{1}$ first, then $x_{1}+y_{1}$ would be the `one-dimensional'
function. Therefore, we define the dimension looking at \emph{sequences}
of functions in $S$. Let $(s_{j})_{j<|S|}$ be a given enumeration
of $S$.
\begin{defn}
A coordinate function $s_{k}\in S$ is \emph{one-dimensional} (resp.\emph{\ two-dimensional},
\emph{zero-dimensional}) if $\td_{\overline{c}r(s_{0},\dots,s_{k-1})}r(s_{k})$
is one (resp.\ two, zero).
\end{defn}
Looking at the values of the functions on the points of $M\cdot r(V)$,
a coordinate function is one-dimensional (resp.\ two-dimensional,
zero-dimensional) if after fixing the values of the functions $s_{0},\dots,s_{k-1}$
on $r(V)$, the image of $s_{k}$ on $r(V)$ is generically a one-dimensional
(resp.\ two-dimensional, zero-dimensional) subset of $K=R^{2}$.
The following remark is now rather trivial, and it shows the better
estimate we were looking for.
\begin{prop}
Let $S$ be a set of coordinate functions of $M\cdot\check{V}$ enumerated
as $(s_{j})_{j<|S|}$. If $S$ contains $k_{1}$ one-dimensional functions,
and $k_{2}$ two-dimensional functions, then $\td_{\overline{c}}S\geq k_{1}+k_{2}$
and $\td_{\overline{c}}(r(S))=k_{1}+2k_{2}$.
\end{prop}
Indeed, this is our desired correction: we manage to prove that $\td_{\overline{c}}S>\td_{\overline{c}}(r(S))/2$
as soon as there are one-dimensional functions, so that we get a better
bound.

\subsection{The one-dimensional case}

We start with a special case where the shape of $M$ does not allow
the functions to be two-dimensional.
\begin{prop}
\label{prop:one-dim-case}If $V$ is simple and
\[
M=\left(\begin{array}{c|c}
N & 0\\
\hline 0 & P
\end{array}\right),
\]
with $N\in\Mc_{k,n}(\Z)$ and $P\in\Mc_{l,n}(\Z)$ two matrices of
ranks $k$ and $l$ respectively, with $k+l>0$, then
\[
\dim M\cdot\check{V}\geq\rank M=k+l
\]
 and if the equality holds then $V$ is perfectly rotund and $k=l=n$.
\end{prop}
In this case, the coordinate functions of $M\cdot r(V)$ cannot be
two-dimensional; in particular, the dimension of $M\cdot\check{V}$
as an algebraic variety is at least the dimension of $M\cdot r(V)$
as a semi-algebraic variety. Therefore, it is sufficient to look at
the semi-algebraic dimension of $M\cdot r(V)$.

In order to do that, we look for a minimal matrix $\check{M}$ such
that the $M\cdot r(V)$ is essentially the projection onto some coordinates
of $r(\check{M}\cdot V)$ (up to a multiplication by a square integer
matrix of maximum rank). By doing this carefully, we can exploit the
simpleness of $V$ to obtain the desired statement.
\begin{lem}
\label{lem:matrix-decomp}There exist matrices $A,N_{0},P_{0},P_{1}$
such that
\begin{itemize}
\item $N_{0},P_{0},P_{1}\in\Mc_{\cdot,n}(\Z)$;
\item $A\in\Mc_{k+l,k+l}(\Z)$ has rank $k+l$ and 
\[
A\cdot\left(\begin{array}{c|c}
N & 0\\
\hline 0 & P
\end{array}\right)=\left(\begin{array}{c|c}
\begin{array}{c}
N_{0}\\
\hline P_{1}
\end{array} & 0\\
\hline 0 & \begin{array}{c}
P_{0}\\
\hline P_{1}
\end{array}
\end{array}\right);
\]

\item the rows of 
\[
\left(\begin{array}{c}
N_{0}\\
\hline P_{0}\\
\hline P_{1}
\end{array}\right)
\]
 are $\Q$-linearly independent.
\end{itemize}
\end{lem}
\begin{proof}
Consider the intersection of the $\Q$-linear spaces generated by
the rows of $N$ and by the rows of $P$. Let $P_{1}$ be the matrix
whose rows form a $\Q$-linear base of this intersection; let $N_{0}$
and $P_{0}$ be the matrices whose rows form $\Q$-linear bases of
the respective spaces over the span of the rows of $P_{1}$. Without
loss of generality, we may assume that all these vectors are in the
$\Z$-module generated by the rows of $N$ and $P$ respectively.

Under these assumptions, there exists two square integer matrices
$A_{0},A_{1}$ of maximum rank such that 
\[
\left(\begin{array}{c|c}
A_{0} & 0\\
\hline 0 & A_{1}
\end{array}\right)\cdot\left(\begin{array}{c|c}
N & 0\\
\hline 0 & P
\end{array}\right)=\left(\begin{array}{c|c}
A_{0}\cdot N & 0\\
\hline 0 & A_{1}\cdot P
\end{array}\right)=\left(\begin{array}{c|c}
\begin{array}{c}
N_{0}\\
\hline P_{1}
\end{array} & 0\\
\hline 0 & \begin{array}{c}
P_{0}\\
\hline P_{1}
\end{array}
\end{array}\right).
\]

Taking as $A$ the leftmost matrix, we find the desired decomposition.\end{proof}
\begin{prop}
\label{prop:one-dim-case-dico}Under the hypothesis of \prettyref{prop:one-dim-case},
we have that
\[
\odim(M\cdot r(V))\geq\rank M,
\]
 and if the equality holds then $V$ is perfectly rotund and either
\begin{itemize}
\item $k=l=n$, or
\item $k+l=n$, and the rows of $N$ and $P$ are $\Q$-linearly independent.
\end{itemize}
\end{prop}
\begin{proof}
We decompose $M$ as in \prettyref{lem:matrix-decomp}: 
\[
M':=A\cdot M=\left(\begin{array}{c|c}
\begin{array}{c}
N_{0}\\
\hline P_{1}
\end{array} & 0\\
\hline 0 & \begin{array}{c}
P_{0}\\
\hline P_{1}
\end{array}
\end{array}\right).
\]
Since $A$ is square of maximum rank, $\odim(M\cdot r(V))=\odim(M'\cdot r(V))$.
Therefore, it is sufficient to study $M'\cdot r(V)$.

We define our $\check{M}$ as 
\[
\check{M}:=\left(\begin{array}{c}
N_{0}\\
\hline P_{0}\\
\hline P_{1}
\end{array}\right)
\]

and we fix the following notation:
\begin{itemize}
\item $n_{0}:=\rank N_{0}$;
\item $p_{0}:=\rank P_{0}$;
\item $p_{1}:=\rank P_{1}$.
\end{itemize}
Note that $\rank N=k=n_{0}+p_{1}$ and that $\rank P=l=p_{0}+p_{1}$.

Let $\Pi$ be the projection matrix such that
\[
\Pi\cdot\left(\begin{array}{c|c}
\check{M} & 0\\
\hline 0 & \check{M}
\end{array}\right)=\Pi\cdot\left(\begin{array}{c|c}
\begin{array}{c}
N_{0}\\
\hline P_{0}\\
\hline P_{1}
\end{array} & 0\\
\hline 0 & \begin{array}{c}
N_{0}\\
\hline P_{0}\\
\hline P_{1}
\end{array}
\end{array}\right)=\left(\begin{array}{c|c}
\begin{array}{c}
N_{0}\\
\hline P_{1}
\end{array} & 0\\
\hline 0 & \begin{array}{c}
P_{0}\\
\hline P_{1}
\end{array}
\end{array}\right)=M'.
\]

In particular, $\Pi\cdot r(\check{M}\cdot V)=M'\cdot r(V)$.

Let us pick a set of algebraically independent coordinates of $\check{M}\cdot V$
in the following way. First, we pick a maximal set $S_{0}$ of algebraically
independent coordinates of $P_{1}\cdot V$, and then we take a maximal
algebraically independent set $S\supset S_{0}$ of coordinates of
$\check{M}\cdot V$.

Let $s_{0}:=|S_{0}|$ and $s_{1}:=|S\setminus S_{0}|$. The simpleness
of $V$ implies that $s_{0}\geq p_{1}$, and if the equality holds
then either $p_{1}=0$ or $V$ is perfectly rotund and $p_{1}=n$.
Similarly, $s_{0}+s_{1}\geq n_{0}+p_{0}+p_{1}$; as before, if the
equality holds then either $n_{0}+p_{0}+p_{1}=0$, which would imply
$M=0$, a contradiction, or $V$ is perfectly rotund and $n_{0}+p_{0}+p_{1}=n$.

Given a finite $\overline{c}$ that defines $\check{V}$, the sets
$r(S_{0})$ and $r(S)$ are both algebraically independent over $\overline{c}$,
and they contain respectively $2s_{0}$ and $2(s_{0}+s_{1})$ coordinate
functions of $r(\check{M}\cdot V)$.

Let us look at the effect of the projection $\Pi$ on $r(S_{0})$
and $r(S\setminus S_{0})$. By construction, each coordinate of $r(S_{0})$
is also a coordinate function of $M'\cdot r(V)$, so $r(S_{0})$ does
not change after the projection $\Pi$.

Let $\phi$ be a function in $S\setminus S_{0}$. Since $S_{0}$ was
maximal, the function $\phi$ must be of the form $\overline{m}\cdot\overline{z}$
or $\overline{w}^{\overline{m}}$ for some row $\overline{m}$ of
either $N_{0}$ or of $P_{0}$. If $\{\phi_{0},\phi_{1}\}=r(\{\phi\})$,
after the projection $\Pi$, exactly one of $\phi_{0}$, $\phi_{1}$
is also a coordinate function of $M'\cdot r(V)$.

This shows that all the coordinate functions in $r(S_{0})$ are also
coordinates of $M'\cdot r(V)$, while exactly half of the coordinates
in $r(S\setminus S_{0})$ are coordinates of $M'\cdot r(V)$. Since
they remain algebraically independent, we have
\begin{eqnarray*}
 & \odim(M\cdot r(V))=\odim(M'\cdot r(V))\geq\\
 & \geq2s_{0}+s_{1}=s_{0}+(s_{0}+s_{1})\geq p_{1}+n_{0}+p_{0}+p_{1}=k+l=\rank M.
\end{eqnarray*}

This is the required inequality. Moreover, if the equality is true,
then we must have had the equality at the previous steps; in particular,
$V$ is perfectly rotund, and either $p_{1}=n$, which implies $k=l=n$,
or $p_{1}=0$ and $n_{0}+p_{0}=n$, which implies that $k+l=n$ and
that the rows of $N$ and $P$ are $\Q$-linearly independent.
\end{proof}
This leaves us with two possible cases. We claim that the second case
does not actually happen.
\begin{prop}
\label{prop:one-dim-case-bad}Under the hypothesis of \prettyref{prop:one-dim-case},
if $k+l=n$ then
\[
\odim(M\cdot r(V))>\rank M.
\]

\end{prop}
The proof of this fact relies on the fact that an algebraic dependence
between two algebraic functions $z$ and $w$ yields an algebraic
dependence between, say, $|z|$ and $|w|$ only in a few exceptional
cases. This is be enough to find an extra algebraically independent
function, proving the strict inequality. Since the argument is topological
in nature and rather different from the current discussion, we postpone
its proof to \prettyref{sub:topological}.
\begin{proof}[Proof of \prettyref{prop:one-dim-case}]
 The coordinate functions of $M\cdot r(V)$ must be all either zero-dimensional
or one-dimensional, because they can only be of the form $\overline{m}\cdot\overline{x}$,
$\overline{p}\cdot\overline{y}$, $\overline{\rho}^{\overline{m}}$,
$\overline{\theta}^{\overline{p}}$. In particular, if $S$ is a transcendence
base of coordinate functions of $M\cdot\check{V}$, we have that $\td_{\overline{c}E(\overline{c})}S\geq\td_{\overline{c}E(\overline{c})}r(S)$.
In geometric terms, the algebraic dimension $\dim M\cdot\check{V}$
is at least the semi-algebraic dimension $\odim M\cdot r(V)$. In
this special case we can actually check that this is an equality.

Therefore, by \prettyref{prop:one-dim-case-dico}, $\dim M\cdot\check{V}\geq\rank M$.
Moreover, if the equality holds, then $V$ is perfectly rotund, and
either $k=l=n$, or the rows of $N$ and $P$ are $\Q$-linearly independent
and $k+l=n$. But by \prettyref{prop:one-dim-case-bad} the latter
condition implies $\dim M\cdot\check{V}>\rank M$, so only the former
case is possible, as desired.
\end{proof}

\subsection{The mixed case}

We continue with another special case.
\begin{prop}
\label{prop:mixed-dim-case}If 
\[
M=\left(\begin{array}{c|c}
N_{0} & 0\\
\hline N_{1} & P_{1}\\
\hline 0 & P_{0}
\end{array}\right),
\]
 with $N_{0}\in\Mc_{k,n}(\Z)$, $P_{0}\in\Mc_{l,n}(\Z)$ and $N_{1},P_{1}\in\Mc_{m,n}(\Z)$
four matrices with the property that
\[
\left(\begin{array}{c}
N_{0}\\
\hline N_{1}
\end{array}\right),\:\left(\begin{array}{c}
P_{0}\\
\hline P_{1}
\end{array}\right)
\]
have ranks $k+m$ and $m+l$ respectively, with $k+l+m>0$, then
\[
\dim M\cdot\check{V}\geq\rank M=k+l+m
\]
 and if the equality holds then $V$ is perfectly rotund and either 
\begin{itemize}
\item $m=0$ and $k=l=n$, or
\item $k=l=0$ and $m=n$.
\end{itemize}
\end{prop}
In this case, the functions coming from the rows of $N_{0}$ and $P_{0}$
are one-dimensional, while the ones coming from the rows of $(N_{1}|P_{1})$
could be two-dimensional, so we may have that $\dim M\cdot\check{V}<\odim M\cdot r(V)$.
However, the arguments of \prettyref{sub:mixed-functions} show that
we may still obtain a meaningful result if we find enough one-dimensional
functions.

We shall use \prettyref{prop:one-dim-case} twice: once on the matrices
$N_{0}$ and $P_{0}$, to find many one-dimensional functions, and
once on a matrix derived from $M$ in order to give a general estimate
on $\odim M\cdot r(V)$. The two estimates together will yield the
result.
\begin{proof}
First, we give a lower bound for $\odim M\cdot r(V)$. Note that
\[
M=\left(\begin{array}{c|c}
N_{0} & 0\\
\hline N_{1} & P_{1}\\
\hline 0 & P_{0}
\end{array}\right)=\left(\begin{array}{c|c}
\begin{array}{c}
\\
\mathrm{Id}_{k+m}\\
\\
\hline 0
\end{array} & \begin{array}{c}
0\\
\hline \\
\mathrm{Id}_{m+l}\\
\\
\end{array}\end{array}\right)\cdot\left(\begin{array}{c|c}
\begin{array}{c}
N_{0}\\
\hline N_{1}
\end{array} & 0\\
\hline 0 & \begin{array}{c}
P_{0}\\
\hline P_{1}
\end{array}
\end{array}\right).
\]
Let us call the last matrix on the right $M'$. The induced map
\[
M'\cdot r(V)\to M\cdot r(V)
\]
 is invertible; indeed, in order to calculate its inverse, it is sufficient
to separate again the real and imaginary parts of the coordinates
coming from the rows of $(N_{1}|P_{1})$.

This implies that
\[
\odim M\cdot r(V)=\odim M'\cdot r(V).
\]

By \prettyref{prop:one-dim-case}, this implies that 
\[
\odim M\cdot r(V)\geq k+2m+l,
\]
 and if the equality holds then $V$ is perfectly rotund and $k+m=m+l=n$.

We now look for many one-dimensional functions. Let us consider the
projection
\[
\Pi\cdot M=\Pi\cdot\left(\begin{array}{c|c}
N_{0} & 0\\
\hline N_{1} & P_{1}\\
\hline 0 & P_{0}
\end{array}\right)=\left(\begin{array}{c|c}
N_{0} & 0\\
\hline 0 & P_{0}
\end{array}\right)=:M''.
\]

The coordinate functions of $M''\cdot r(V)$ can only be zero-dimensional
or one-dimensional. By \prettyref{prop:one-dim-case} again, if $M''\neq0$
then $\odim(M''\cdot r(V))\geq k+l$, and if the equality holds then
$V$ is perfectly rotund and $k=l=n$.

Now, let $S_{0}$ be a maximal set of algebraically independent coordinate
functions of $M''\cdot r(V)$ and let $S\supset S_{0}$ be a maximal
set of algebraically independent coordinate functions of $M\cdot r(V)$.

By \prettyref{sub:mixed-functions}, if $k_{1}$ is the number of
one-dimensional functions of $S$, and $k_{2}$ is the number of its
two-dimensional functions in any ordering, then $|S|\geq k_{1}+k_{2}$
and $\td_{\overline{c}}r(S)=k_{1}+2k_{2}$, where $\overline{c}$
is a finite set that defines $\check{V}$.

But $\td_{\overline{c}}r(S)$ is the semi-algebraic dimension of $M\cdot r(V)$;
by the above argument,
\[
k_{1}+2k_{2}\geq k+2m+l.
\]

Moreover, if we enumerate the functions in $S$ starting with the
ones in $S_{0}$, then clearly $k_{1}\geq|S_{0}|$, so
\[
k_{1}\geq k+l.
\]

Summing the two inequalities, we have
\[
2|S|\geq2(k_{1}+k_{2})\geq2k+2m+2l=2(k+m+l)=2\rank M,
\]
 as desired. Moreover, if the equality holds, then $V$ is perfectly
rotund, and either $m=0$ and $k=l=n$ or $k=l=0$ and $m=n$.
\end{proof}

\subsection{The general case}

In order to conclude the theorem, it is sufficient to show that any
matrix $M$ can be actually be reduced to the shape required by \prettyref{prop:mixed-dim-case}.
\begin{lem}
\label{lem:matrix-reduce} Let $M\in\Mc_{p,2n}(\Z)$ be a matrix of
rank $0<p\leq2n$. There is a matrix $A\in\Mc_{p,p}(\Z)$ of maximum
rank such that
\[
A\cdot M=\left(\begin{array}{c|c}
N_{0} & 0\\
\hline N_{1} & P_{1}\\
\hline 0 & P_{0}
\end{array}\right)
\]
 with $N_{0}\in\Mc_{k,n}(\Z)$, $P_{0}\in\Mc_{l,n}(\Z)$ and $N_{1},P_{1}\in\Mc_{m,n}(\Z)$
four matrices such that
\[
\left(\begin{array}{c}
N_{0}\\
\hline N_{1}
\end{array}\right),\:\left(\begin{array}{c}
P_{0}\\
\hline P_{1}
\end{array}\right)
\]
have ranks $k+m$ and $m+l$ respectively, and $k+m+l=p$.\end{lem}
\begin{proof}
It is an elementary result of linear algebra that one can find a matrix
$A_{0}\in\Mc_{p,p}(\Z)$ of maximum rank such that 
\[
A_{0}\cdot M=\left(\begin{array}{c|c}
N & Q\\
\hline 0 & P
\end{array}\right),
\]
 with $N$ and $P$ both with $\Q$-linearly independent rows.

Similarly to what we have done before, let $Q_{0}$ be a matrix whose
rows form a $\Q$-linear base of the intersection of the span of the
rows of $Q$ and the span of the rows of $P$. Let $P_{1},P_{2}$
be two matrices whose rows form a $\Q$-linear base of the span of
the rows of $Q$, $P$ respectively over the span of the rows of $Q_{0}$.
We define $k$ as the number of rows of $Q_{0}$, $l$ as the number
of rows of $P$ and $m$ as the number of the rows of $P_{1}$.

Up to multiplying the new matrices by a non-zero integer, we can find
two square integer matrices $A_{1},A_{2}$ of maximum rank such that
\[
\left(\begin{array}{c|c}
A_{1} & 0\\
\hline 0 & A_{2}
\end{array}\right)\cdot\left(\begin{array}{c}
Q\\
\hline P
\end{array}\right)=\left(\begin{array}{c}
A_{1}\cdot Q\\
\hline A_{2}\cdot P
\end{array}\right)=\left(\begin{array}{c}
Q_{0}\\
\hline P_{1}\\
\hline Q_{0}\\
\hline P_{2}
\end{array}\right).
\]

Finally, let us apply the matrix
\[
A_{3}:=\left(\begin{array}{c|c|c|c}
\mathrm{Id}_{k} & 0 & -\mathrm{Id}_{k} & 0\\
\hline 0 & \mathrm{Id}_{m} & 0 & 0\\
\hline 0 & 0 & \mathrm{Id}_{k} & 0\\
\hline 0 & 0 & 0 & \mathrm{Id}_{l-k}
\end{array}\right);
\]
 we have 
\[
A_{3}\cdot\left(\begin{array}{c}
Q_{0}\\
\hline P_{1}\\
\hline Q_{0}\\
\hline P_{2}
\end{array}\right)=\left(\begin{array}{c}
0\\
\hline P_{1}\\
\hline Q_{0}\\
\hline P_{2}
\end{array}\right),\quad A_{3}\cdot\left(\begin{array}{c}
A_{1}\cdot N\\
\hline 0
\end{array}\right)=\left(\begin{array}{c}
A_{1}\cdot N\\
\hline 0
\end{array}\right).
\]

Putting all together, we have
\[
\left(A_{3}\cdot\left(\begin{array}{c|c}
A_{1} & 0\\
\hline 0 & A_{2}
\end{array}\right)\cdot A_{0}\right)\cdot M=\left(\begin{array}{c|c}
A_{1}\cdot N & \begin{array}{c}
0\\
\hline P_{1}
\end{array}\\
\hline 0 & \begin{array}{c}
Q_{0}\\
\hline P_{2}
\end{array}
\end{array}\right).
\]

This shows the required decomposition.
\end{proof}
We have now enough tools to prove the theorem.
\begin{proof}[Proof of~\prettyref{thm:rotund}]
 By \prettyref{lem:matrix-reduce}, there is a square matrix $A$
of maximum rank such that $ $$A\cdot M$ has the shape required in
\prettyref{prop:mixed-dim-case}. Since $A$ is square and has maximum
rank, then $\dim M\cdot\check{V}=\dim A\cdot M\cdot\check{V}$, while
$\rank A\cdot M=\rank M$. But by \prettyref{prop:mixed-dim-case},
\[
\dim M\cdot\check{V}=\dim A\cdot M\cdot\check{V}\geq\rank A\cdot M=\rank M,
\]
 and if the equality holds, then $V$ is perfectly rotund and either
$A\cdot M$ is square of maximum rank, or $A\cdot M=(N'|P')$ with
$N'$ and $P'$ square of maximum rank. But this clearly implies that
either $M$ has rank $p=2n$ or that $M$ has itself the shape $M=(N|P)$
with $N$ and $P$ square matrices of maximum rank, as desired.
\end{proof}

\subsection{\label{sub:topological}A topological argument}

It remains to prove \prettyref{prop:one-dim-case-bad}. We shall use
a topological argument about algebraic curves. Given two non-constant
functions $z$, $w$ on an algebraic curve, we look at their local
behaviour to show that certain pairs, e.g., $|z|$ and $|w|$, are
algebraically independent unless there is a strong reason to (in the
example, $z$ and $w$ would be multiplicatively dependent over the
constants). This is enough to obtain the result.

The trick is to look at the two functions around a zero of $z$. Locally,
the functions behave like $x\mapsto x^{d}$, where $d$ is the ramification
degree, up to analytic equivalence. We shall use this information
to show that certain closed paths must have non-trivial homotopy class;
for example, if we use a common zero of $z$ and $w$ (if it exists),
this implies that $\Theta(w)$ is surjective on $\Sb^{1}$ on the
paths where $|z|$ is constant. The information on the values will
be sufficient to deduce the algebraic independence; in the example,
it shows that $|z|$ and $\Theta(w)$ are algebraically independent.
\begin{lem}
\label{lem:alg-module-curves}Let $\Cc$ be an irreducible curve defined
over some $\overline{c}$ closed under $\sigma$ and let $z,w$ be
non-constant algebraic functions on some absolutely irreducible component
of $\Cc$.

The function $|w|$ (or $\Theta(w)$) is algebraically independent
from $\Re(z)$, $\Im(z)$, or $\Theta(z)$ (resp.\ $|z|$) over $\overline{c}$.
Moreover, if $z$ and $w$ are multiplicatively independent over $\acl(\overline{c})^{\times}$,
then $|w|$ (or $\Theta(w)$) is algebraically independent from $|z|$
(resp.\ $\Theta(z)$) over $\overline{c}$.\end{lem}
\begin{proof}
First of all, we may assume that we are working over $\R$. Indeed,
once we take $\Cc$, $z$ and $w$ in a definable family over $\emptyset$,
the first part of the theorem is clearly expressible as a conjunction
of first-order formulas. Moreover, the degrees of $z$ and $w$ are
bounded in terms of the family.

For the second part, let us assume that the functions are multiplicatively
dependent over $\acl(\overline{c})^{\times}$, i.e., that $z^{k}w^{l}=c\in\acl(\overline{c})^{\times}$
for some $k,l\in\N^{\times}$. Since $\Cc$ is absolutely irreducible,
we may assume that $k$ and $l$ are coprime. But $z^{k}w^{l}$ is
constant if and only if
\[
k\divs(z)+l\divs(w)=0,
\]
where $\divs(z)$ and $\divs(w)$ are the divisors of the two functions.
By elementary theory of curves, there is at least one non-zero coefficient
in $\divs(z)$ whose norm is at most $\deg(z)$ (similarly for $\divs(w)$).
This puts a bound on the size of $k$ and $l$. Therefore, the multiplicative
independence in a given family is witnessed by a single first order
formula, so that the second part of the theorem may be expressed as
a conjunction of first-order formulas. Therefore, we may reduce to
$\R$.

Now, let $P\in\Cc$ be a zero of $z$. Without loss of generality,
we may assume that $P$ is non-singular, otherwise we replace $\Cc$
with a suitable blow-up. Let $U\subset\Cc$ be a neighbourhood of
$P$ in $\Cc$. We take $U$ small enough so that $z(U)$ and $w(U)$
are open sets in $\P_{1}(\C)$, and such that the maps $z:U\setminus\{P\}\to z(U)\setminus\{0\}$
and $w:U\setminus\{P\}\to w(U)\setminus\{w(P)\}$ are covering maps.
We may further assume that $U$ contains no zeroes or poles of $z,w$
except for $P$, so that $\Theta(z)$ and $\Theta(w)$ are well-defined
and continuous on $U\setminus\{P\}$.

We split the proof into three cases. Let $d$ and $e$ be the two
ramification degrees of $z$ and $w$ at $P$.

\textbf{$|w|$, $\Theta(w)$, $\Re(w)$ and $\Im(w)$ are not in $\acl(|z|\overline{c})$.}
Let $\gamma:[0,1]\to z(U)\setminus\{0\}$ be the closed path around
$0$ defined by $|z|=c$ for some small enough $c>0$. If we follow
the path $d$ times, we may lift it to a closed path $\eta$ around
$P\in U$. Since the homotopy class of $\gamma$ in $z(U)\setminus\{0\}$
is non-trivial, the same holds for $\eta$ in $U\setminus\{P\}$.
In particular, its image $w_{*}\eta$ through $w$ is also a closed
path around $w(P)$ with non-trivial homotopy class in $w(P)\setminus\{w(P)\}$.

This immediately implies that whenever we fix $|z|$, $\Re(w)$ and
$\Im(w)$ vary in two proper closed intervals around $\Re(w(P))$,
$\Im(w(P))$ respectively; moreover, if $w(P)\neq0,\infty$, $\Theta(w)$
takes values in a proper closed arc around $\Theta(w(P))$, and if
$w(P)=0,\infty$, its image is the whole $\Sb^{1}$. In particular,
they must be algebraically independent over $\overline{c}$.

Let us study what happens to $|w|$. Let us suppose that $w(P)\neq0,\infty$;
clearly, if $c$ is small enough, the path $w_{*}\eta$ around $w(P)$
is far from $0$. This shows that $|w|$ must vary in a proper closed
interval around $|w(P)|$, and it is therefore algebraically independent
from $|z|$ over $\overline{c}$.

If $w$ has a zero or a pole at $P$, we may replace $w$ with $w'=z^{k}w^{l}$,
for some integers $k,l\in\N^{\times}$ so that the zeroes and the
poles at $P$ balance out. By the hypothesis on $z$ and $w$, $w'$
is non-constant and we may apply the previous argument to show that
$|z|$ and $|w'|$ are algebraically independent over $\overline{c}$.
Since $|w|^{l}=|z|^{-k}|w'|$, this implies that $|z|$ and $|w|$
are algebraically independent over $\overline{c}$ as well.

\textbf{$\Theta(w)$ is not in $\acl(\Theta(z)\overline{c})$.} As
for $|z|$ and $|w|$, we may assume that $w(P)\neq0,\infty$ by replacing
$w$ with a suitable multiplicative combination $z^{k}w^{l}$. Again,
let $\gamma$ be the a closed path around $0\in z(U)$ defined by
$|z|=c$ and let $\eta$ be the lift of $\gamma$ traversed $d$ times.
As before, the range of $\Theta(w)$ is a proper closed arc around
$\Theta(w(P))$.

As $c$ tends to $0$, $\eta$ converges uniformly to $w(P)$, and
since $w(P)\neq0,\infty$, this implies that the range of $\Theta(w)$
converges uniformly to the point $\Theta(w(P))$. If $\Theta(w)$
were dependent on $\Theta(z)$, it would be locally a function of
$\Theta(z)$, which means that it should be actually constantly $\Theta(w(P))$,
a contradiction.

\textbf{$\Theta(z)$ is not in $\acl(\Re(z)\overline{c})$ or in $\acl(\Im(z)\overline{c})$.}
Let us suppose by contradiction that $\Theta(z)$ is algebraically
dependent on $\Re(w)$ over $\overline{c}$. This implies that whenever
we lift a ``vertical'' path in $w(U)\setminus\{w(P)\}$, i.e., one
where $\Re(w)$ is constant, the function $\Theta(z)$ is constant
on the lifted path.

Let $\gamma$ be a closed square path around $w(P)\in w(U)$ with
two vertical sides (note that we are now lifting a path in $w(U)$
rather than in $z(U)$). Let us write $\gamma$ as the concatenation
$\gamma_{1}*\gamma_{2}*\gamma_{3}*\gamma_{4}$, where $\gamma_{1}$,
$\gamma_{3}$ are vertical and $\gamma_{2}$, $\gamma_{4}$ are horizontal.
We may traverse the path $e$ times and lift it to a closed path $\eta$
around $P$ with non-trivial homotopy class, and the image through
$z$ will be a path around $0$ with non-trivial homotopy class in
$z(U)\setminus\{0\}$. In particular, the image through $\Theta(z)$
in $\Sb^{1}$ is a path with non-trivial homotopy class.

On the other hand, the fact that $\Theta(z)$ is constant on the vertical
sides implies that it is constant on the lifts of $\gamma_{1}$ and
$\gamma_{3}$. Let us call $\theta_{1}$ and $\theta_{3}$ the values
on the lifts that appear when we traverse $\gamma$ the first time.
When we lift $\gamma_{2}$ for the first time, we see that $\Theta(z)$
describes a path $\epsilon$ in $\Sb^{1}$ from $\theta_{1}$ to $\theta_{3}$.

However, we can connect the points of $\gamma_{2}$ and $\gamma_{4}$
with vertical lines except for the two points with $\Re(w)=\Re(w(P))$.
This implies that $\Theta(z)$ must take the same values on the lift
of the two paths, except possibly in two points; but then it takes
the same values everywhere by continuity. In particular, the function
$\Theta(z)$ describes exactly the path $\epsilon^{-1}$ on the first
lift $\gamma_{4}$, going from $\theta_{3}$ to $\theta_{1}$. This
implies that the image through $\Theta(z)$ of the path $\eta$ has
\emph{trivial} homotopy class in $\Sb^{1}$, a contradiction.

The same holds for $\Im(w)$ (e.g., by replacing $w$ with $iw$,
or repeating the same argument with $\Im(w)$ in place of $\Re(w)$).

The remaining cases can be recovered by swapping $z$ and $w$ and
using the exchange principle.
\end{proof}
The above statement can be now generalised to varieties of higher
dimension. We say that $z$ and $w$ are interalgebraic over $\overline{c}$
if $z\in\acl(\overline{c}w)$ and $w\in\acl(\overline{c}z)$.
\begin{lem}
\label{lem:alg-module-general}Let $V\subset\G^{n}$ be an absolutely
irreducible algebraic variety defined over some $\overline{c}$ closed
under $\sigma$. Let $B$ be a set of algebraically independent coordinate
functions of $V$, and let $w$ be a function in $\acl(B\overline{c})$.

The function $|w|$ (or $\Theta(w)$) cannot be interalgebraic with
$\Re(z)$, $\Im(z)$, or $\Theta(z)$ (resp.\ $|z|$) over $r(B\setminus\{z\})\cup\overline{c}$
for any $z\in B$.

Moreover, if $B$ contains only multiplicative functions, and $|w|\in\acl(\{|z|\,:\, z\in B\}\cup\overline{c})$
or $\Theta(w)\in\acl(\{\Theta(z)\,:\, z\in B\}\cup\overline{c})$,
then $w$ is multiplicatively dependent on $B$ modulo constants.\end{lem}
\begin{proof}
First of all, we may assume that $B$ is minimal such that $w\in\acl(B\overline{c})$.
Indeed, suppose that $B'\subsetneq B$ and $w\in\acl(B'\overline{c})$.
The functions $|w|$ and $\Theta(w)$ are in $\acl(r(B')\overline{c})$
and they cannot be interalgebraic with the components of $z$, so
that the first part of the thesis holds trivially.

Moreover, if $|w|$ depends only on the moduli of the functions in
$B$, then it actually depends only on the moduli of the functions
in $B'$, because $r(B)$ is algebraically independent over $\overline{c}$.
The same holds for $\Theta(w)$. Therefore, the thesis for $B'$ implies
the thesis for $B$.

Let us assume that $B$ is minimal. We shall work by induction on
$|B|$, the case $|B|=0$ being trivial.

Since $B$ is minimal, for any $z\in B$ the functions satisfy the
equation
\[
p(z,w)=0
\]
 for some irreducible polynomial with coefficients in $\overline{c}B\setminus\{z\}$
such that both $z$ and $w$ appear. We specialise the variables in
$r(B)$ generically over $\overline{c}$, so that the polynomial define
an affine algebraic curve of which we take a projective model $\Cc$.

\prettyref{lem:alg-module-curves} implies immediately the first part
of the thesis.

If the hypothesis of the second part is satisfied, then \prettyref{lem:alg-module-curves}
implies that there must be two integers $k,l$ not both zero such
that the specialisation of $z^{k}w^{l}$ is constant. But the specialisation
was generic, hence $z^{k}w^{l}\in\acl(\overline{c}B')$ with $B':=B\setminus\{z\}$.

But now $|z^{k}w^{l}|$ (resp.\ $\Theta(z^{k}w^{l})$) only depends
on the moduli (resp.\ phases) of the functions in $B$, hence of
the functions in $B'$. The inductive hypothesis implies that $z^{k}w^{l}$
is multiplicatively dependent on $B'$ modulo constants; therefore,
$w$ is multiplicatively dependent on $B=B'\cup\{z\}$ modulo constants,
as desired.
\end{proof}
The above lemma suggests where to find some extra algebraically independent
functions for the proof of \prettyref{prop:one-dim-case-bad}. Indeed,
let us assume that $B$ is a transcendence base of the coordinate
functions of $\check{M}\cdot V$, and that $\overline{w}^{\overline{q}}$
is a coordinate function of $\check{M}\cdot V$ not in $B$ which
is interalgebraic with some $\overline{m}\cdot\overline{z}\in B$
over $B\setminus\{\overline{m}\cdot\overline{z}\}$.

Under the hypothesis of \prettyref{prop:one-dim-case-bad}, exactly
one function between $\overline{m}\cdot\overline{x}$ and $\overline{m}\cdot\overline{y}$
is a coordinate function of $M\cdot r(V)$, say $\overline{m}\cdot\overline{x}$,
and similarly say that $\overline{\rho}^{\overline{q}}$ is the only
component of $\overline{w}^{\overline{q}}$ appearing on $M\cdot r(V)$.
The lemma implies that $\overline{m}\cdot\overline{x}$ and $\overline{\rho}^{\overline{q}}$
are algebraically independent over $\overline{c}r(B\setminus\{\overline{m}\cdot\overline{z}\})$;
in this way, we gain a new independent function that was not in $r(B)$.

When the above choice of functions is possible, this proves the statement
of \prettyref{prop:one-dim-case-bad}. With a little abuse of notation,
we denote by $\pi_{a}$ the projection of $\G^{n}\to\G_{a}^{n}$ and
by $\pi_{m}$ the projection $\G^{n}\to\G_{m}^{n}$; the above choice
is possible exactly when the projection $\pi_{m\restriction V}$ does
not split, i.e., when $V$ is not the cartesian product of $\pi_{a}(V)$
and $\pi_{m}(V)$. In the non-split case, the previous argument yields
a proof, while we need a slightly different version for the split
case.

Recall that we are in the case 
\[
M=\left(\begin{array}{c|c}
N & 0\\
\hline 0 & P
\end{array}\right)
\]
with $M$ of rank $n$. If we define 
\[
\check{M}=\left(\begin{array}{c}
N\\
\hline P
\end{array}\right),
\]
 then for each coordinate function $\phi$ of $\check{M}\cdot V$,
exactly one function of $r(\{\phi\})$ is also a coordinate of $M\cdot r(V)$.
\begin{prop}
\label{prop:one-dim-case-bad-additive}Under the hypothesis of \prettyref{prop:one-dim-case-bad},
if $V$ does not split as the product of the two projections $\pi_{a}(V)\subset\G_{a}^{n}$
and $\pi_{m}(V)\subset\G_{m}^{n}$, then 
\[
\odim(M\cdot r(V))>\rank M.
\]
\end{prop}
\begin{proof}
We may assume that $\rank M=\rank\check{M}=n$, otherwise the conclusion
would already follow from \prettyref{prop:one-dim-case-dico}.

If $V$ is not the product of $\pi_{a}(V)$ and $\pi_{m}(V)$, then
the same is true for $\check{M}\cdot V$ because we are assuming that
$\check{M}$ has rank $n$. In particular, if $B$ is any transcendence
basis over $\overline{c}$ of coordinate functions of $\check{M}\cdot V$,
we can find a function $\phi$ on $\pi_{a}(\check{M}\cdot V)$ and
a function $\psi$ on $\pi_{m}(\check{M}\cdot V)$, with one of them
in $B$, such that they are interalgebraic over $\overline{c}B\setminus\{\phi,\psi\}$.
Up to replacing $\psi$ with $\phi$ in $B$, using the exchange property,
we may assume that $\phi\in B$.

Let $B'$ be the subset of $r(B)$ of the functions appearing as coordinates
of $M\cdot r(V)$. If $\{\phi_{1},\phi_{2}\}=r(\{\phi\})$ and $\{\psi_{1},\psi_{2}\}=r(\{\psi\})$,
let $\phi_{i}$ and $\psi_{j}$ be the two functions appearing as
coordinates of $M\cdot r(V)$.

By the first consequence of \prettyref{lem:alg-module-general}, $\phi_{i}$
and $\psi_{j}$ are algebraically independent over $\overline{c}r(B\setminus\{\phi\})$.
This implies in particular that $\phi_{i}$ and $\psi_{j}$ are algebraically
independent over $\overline{c}B'\setminus\{\phi_{i}\}$, which in
turn means that $B'\cup\{\psi_{j}\}$ is algebraically independent
over $\overline{c}.$ But then 
\[
\odim(M\cdot r(V))\geq|B'\cup\{\psi_{j}\}|>|B'|\geq\rank M.
\]

\end{proof}
In order to analyse the case in which $\pi_{m}$ splits we need the
second part of \prettyref{lem:alg-module-general}.
\begin{prop}
\label{prop:one-dim-case-bad-mult} Under the hypothesis of \prettyref{prop:one-dim-case-bad},
if $V$ is the product of the two projections $\pi_{a}(V)\subset\G_{a}^{n}$
and $\pi_{m}(V)\subset\G_{m}^{n}$, then 
\[
\odim(M\cdot r(V))>\rank M.
\]
\end{prop}
\begin{proof}
Since $V$ splits, then $\check{M}\cdot V$ splits as well, and we
have $\dim\check{M}\cdot V=\dim\pi_{a}(\check{M}\cdot V)+\dim\pi_{m}(\check{M}\cdot V)$.

Again, we may assume that $V$ is perfectly rotund and that $\rank M=n$,
otherwise the conclusion would already follow from \prettyref{prop:one-dim-case-dico}.

In particular, we may assume that 
\[
\dim\pi_{a}(\check{M}\cdot V)+\dim\pi_{m}(\check{M}\cdot V)=\rank M=n.
\]

Moreover,
\[
\odim\pi_{a}(M\cdot r(V))+\odim\pi_{m}(M\cdot r(V))=\odim M\cdot r(V).
\]

Note that $\odim\pi_{a}(M\cdot r(V))\geq\dim\pi_{a}(\check{M}\cdot V)$.
Therefore, it shall be sufficient to prove that \foreignlanguage{english}{$\odim\pi_{m}(M\cdot r(V))>\dim\pi_{m}(\check{M}\cdot V)$.}

By absolute freeness, the two projections cannot have dimension $0$,
so we must have $0<\dim\pi_{m}(\check{M}\cdot V)<n$. Let $B$ be
a transcendence basis over $\overline{c}$ of coordinate functions
of $\pi_{m}(\check{M}\cdot V)$, and let $\psi$ be a coordinate not
in $B$. Let $B_{0}\subset B$ be a minimal set such that $\psi\in\acl(\overline{c}B_{0})$.
Again, let $B'$ be the set of the functions in $r(B)$ appearing
as coordinates of $M\cdot r(V)$, and let $B_{0}':=B'\cap r(B_{0})$.

Let us suppose that $|\psi|$ is the member of $r(\{\psi\})$ appearing
as a coordinate of $M\cdot r(V)$. First, we claim that $B_{0}'\cup\{|\psi|\}$
is algebraically independent over $\overline{c}$.

By the first part of \prettyref{lem:alg-module-general}, $|\psi|$
is not interalgebraic with $\Theta(z)$ over $\overline{c}r(B_{0}\setminus\{z\})$
for any $z\in B_{0}$. In particular, $|\psi|$ is not interalgebraic
with $\Theta(z)$ over $\overline{c}B_{0}'\setminus\{|z|,\Theta(z)\}$.
If $\Theta(z)\in B_{0}'$ for some for some $z\in B_{0}$, then actually
$B_{0}'\cup\{|\psi|\}$ is algebraically independent over $\overline{c}$.
On the other hand, if $|z|\in B_{0}'$ for all $z\in B$, the absolute
freeness of $V$ and the second part of \prettyref{lem:alg-module-general}
imply that $B_{0}'\cup\{|\psi|\}$ is algebraically independent as
well.

In particular, since the functions in $B\setminus B_{0}$ are algebraically
independent over $\acl(\overline{c}B_{0})$, and in particular the
functions in $B'\setminus B_{0}'$ are independent over $\acl(B_{0}'\cup\{|\psi|\})$,
this implies that $B'\cup\{|\psi|\}$ is algebraically independent
over $\overline{c}$, and therefore 
\[
\odim\pi_{m}(M\cdot r(V))\geq|B'\cup\{|\psi|\}|>|B'|=\dim\pi_{m}(\check{M}\cdot V).
\]

If $\Theta(\psi)$ is the member of $r(\{\psi\})$ appearing as a
coordinate of $M\cdot r(V)$, it is sufficient to repeat the above
argument replacing every occurrence of $|\cdot|$ with $\Theta(\cdot)$
and every occurrence of $\Theta(\cdot)$ with $|\cdot|$. In particular,
\begin{eqnarray*}
 & \odim M\cdot r(V)=\odim\pi_{a}(M\cdot r(V))+\odim\pi_{m}(M\cdot r(V))>\\
 & \dim\pi_{a}(\check{M}\cdot V)+\dim\pi_{m}(\check{M}\cdot V)=\dim\check{M}\cdot V=\rank M,
\end{eqnarray*}
 as desired.
\end{proof}
This is enough to prove \prettyref{prop:one-dim-case-bad}.
\begin{proof}[Proof of~\prettyref{prop:one-dim-case-bad}]
 This is a corollary of \prettyref{prop:one-dim-case-bad-additive}
and \prettyref{prop:one-dim-case-bad-mult}.
\end{proof}
This completes the proof of \prettyref{thm:rotund}.
\begin{rem}
Note that the absolute (multiplicative) freeness in the proof of is
crucial for the split case of \prettyref{prop:one-dim-case-bad} to
hold, as shown by the following example.

Let $M\in\Mc_{k,n}(R)$, $N\in\Mc_{l,n}(\Z)$ be two matrices of maximum
rank with $k+l=n$. The variety defined by
\[
V_{M,N}:=\{\intv{\overline{z}}{\overline{w}}\in\G^{n}\,:\, M\cdot\overline{z}=\overline{0},\overline{w}^{N}=\overline{1}\}
\]
 has the property that
\[
\odim\left(\begin{array}{c|c}
\mathrm{Id} & 0\end{array}\right)\cdot V_{M,N}=k+l=n,
\]
 and indeed
\[
\dim\left(\begin{array}{c|c}
\mathrm{Id} & 0\end{array}\right)\cdot\check{V}_{M,N}=k+l=n.
\]

Note that with a careful choice of $M$ and $N$, we may even take
$V_{M,N}$ perfectly rotund and absolutely additively free.\end{rem}

%% file: RootsGRes.tex
\section{\label{sec:roots-g-res}Solutions and roots on the $\G$-restrictions}

As anticipated in \prettyref{sec:roots}, we need some kind of alternative
version of \prettyref{prop:solutions-in-roots} to work with dense
sets of real generic solutions. In order to do this, we produce a
suitable generalisation taking into accounts the $\G$-restrictions.
First of all, we introduce some other technical definitions.
\begin{defn}
An \emph{open subsystem} of $\Rc(V)$ is a collection of subsets $\Uc:=\{\Uc(W)\}_{W\in\Rc(V)}$
such that $\Uc(W)\subset W$ is open in $W$ w.r.t.\ the order topology.
\end{defn}
We extend the usual set-theoretic operations to the open subsystems:
the union of two or more open subsystems is the subsystem of the union
of the respective open sets, one subsystem $\Uc$ is contained in
another $\Uc'$ if for each $W\in\Rc(V)$ we have $\Uc(W)\subset\Uc'(W)$,
etc..
\begin{defn}
\label{def:dense-solved}A system $\Rc(V)$, for $V$ defined over
$\overline{c}$, with $\overline{c}$ closed under $\sigma$, is \emph{densely
solved} w.r.t. an open subsystem $\Uc$ if for all $W\in\Rc(V)$ there
is an infinite set of solutions of $W$, really algebraically independent
over $\acl(\overline{c})$, which is dense in $\Uc(W)$ w.r.t.\ the
order topology.

We say that $\Rc(V)$ is \emph{densely solved in $K_{E}$} if it is
densely solved in $K_{E}$ w.r.t.\ the trivial open subsystem $\Rc(V)$.
\end{defn}
Again, this definition does not depend on $\overline{c}$. Indeed,
if we use another $\overline{d}$ then \prettyref{prop:indep-up-to-finite}
implies that a set of algebraically independent points remains algebraically
independent when passing from $\overline{c}$ to $\overline{d}$,
up to removing a finite set. As there are no isolated points in any
of the varieties $W\in\Rc(V)$, the set of solutions remains dense
w.r.t\ an open subsystem after removing a finite set.
\begin{prop}
If $W\in\Rc(V)$, then $\check{W}\in\Rc(\check{V})$.\end{prop}
\begin{proof}
There is $q\in\N^{\times}$ such that $q\cdot W=V$. Taking the realisation
we obtain $q\cdot r(W)=r(V)$; taking the Zariski closure, $q\cdot\check{W}=\check{V}$.\end{proof}
\begin{prop}
If $V$ is Kummer-generic, then $\Rc(V)$ is densely solved if and
only if $\Rc(W)$ is densely solved for some $W\in\Rc(V)$.\end{prop}
\begin{proof}
Since $V$ is Kummer-generic, each variety $W'\in\Rc(V)$ is of the
form $q\cdot W''$ for a $W''\in\Rc(W)$. If $\Rc(W)$ is densely
solved, then $W''$ has a dense set of really algebraically independent
solutions, which implies that $q\cdot W''=W'$ has one as well. In
particular, $\Rc(V)$ is densely solved.
\end{proof}
The following statement is the `right' version of \prettyref{prop:solutions-in-roots}
for densely solved systems.
\begin{prop}
\label{prop:dense-solutions-in-roots} Let $V\subset\G^{n}$ be a
Kummer-generic rotund variety and $K_{E}$ such that $\sigma\circ E=E\circ\sigma$.
Let $N,P\in\Mc_{n,n}(\Z)$ be two square integer matrix of maximum
rank, $\overline{z}\in\dom(E)^{n}$, and $X:=(N|P)\cdot\check{V}\oplus\zEz{\overline{z}}$.

If $\dim X=\dim V$ and $X$ is Kummer-generic, then there exists
an open subsystem $\Uc$ of $\Rc(X)$ such that $\Rc(V)$ is densely
solved if and only if $\Rc(X)$ is densely solved w.r.t.\ $\Uc$.\end{prop}
\begin{proof}
For each $W\in\Rc(V)$ we have $\check{W}\in\Rc(\check{V})$; in particular,
there is a $q\in\N^{\times}$ such that $M\cdot\check{W}\oplus\zEz{\frac{1}{q}\overline{z}}\in\Rc(X)$.
Let $Y$ be $M\cdot\check{W}\oplus\zEz{\frac{1}{q}\overline{z}}$.
Since $X$ is Kummer-generic, as $W$ varies in $\Rc(V)$, the variety
$Y$ takes all the possible values in $\Rc(X)$.

For brevity, let $\psi_{W}:W\to Y$ be the map 
\[
\psi_{W}(Q):=\left(\left(\begin{array}{c|c}
N & P\end{array}\right)\cdot r(Q)\right)\oplus\zEz{\frac{1}{q}\overline{z}}.
\]

Let $\tilde{N},\tilde{P}$ be two integer matrices such that $\tilde{N}\cdot N=\tilde{P}\cdot P=k\cdot\mathrm{Id}$
for some $k\in\N^{\times}$. We define $\tilde{\psi}_{W}:\psi_{W}(W)\to W$
as
\[
\tilde{\psi}_{W}(R):=\left(\begin{array}{c|c}
\tilde{N} & \tilde{P}\end{array}\right)\cdot r\left(R\oplus\zEz{-\frac{1}{q}\overline{z}}\right).
\]

Both are semi-algebraic maps between $r(W)$ and $r(Y)$. We claim
that $\psi_{W}$ is finite-to-one onto its image. Indeed, if $\psi_{W}(Q)$
is a point in the image, then 
\[
\tilde{\psi}_{W}(\psi_{W}(Q))=k\cdot r(Q)=r(k\cdot Q).
\]
Since there are finitely many points $Q'$ such that $k\cdot Q'=k\cdot Q$,
and $r$ is injective, the map $\psi_{W}$ is finite-to-one.

This implies that $\odim(\psi_{W}(W))=\odim(Y)$. In particular, there
is an open subset $U_{Y}$ of $Y$ such that $U_{Y}\subset\psi_{W}(W)$
and 
\[
\odim(\psi_{W}(W)\setminus U_{Y})<\odim(\psi_{W}(W)).
\]
In particular, $\psi_{W}^{-1}(U_{Y})$ is an open subset of $V$ such
that 
\[
\odim(W\setminus\psi_{W}^{-1}(U_{Y}))<\odim(W).
\]

For each $Y$, we have selected an open subset $U_{Y}$. We claim
that the resulting open subsystem $\Uc$ is the desired one.

From now on, let $\overline{c}$ be a defining tuple for $\check{V}$
containing $r(\zEz{\overline{z}})$ and closed under $\sigma$. In
particular, $\overline{c}$ also defines both $V$ and $X$.

The left-to-right direction is clear: since $\psi_{W}$ is continuous,
semi-algebraic, defined over $\overline{c}$, and finite-to-one, it
sends really algebraically independent dense sets to really algebraically
independent sets dense in the image; in particular, the image of the
solutions in $W$ through $\psi_{W}$ will be a really algebraically
independent set dense in $U_{Y}$. As the subsystem $\Uc$ is composed
exactly by the $U_{Y}$'s, the conclusion follows.

For the right-to-left we proceed as above. The map $\tilde{\psi}_{W}$,
for $Q\in U_{Y}$, is such that $\tilde{\psi}_{W}\circ\psi_{W}$ is
exactly $k\cdot\mathrm{Id}$. Hence it is a finite-to-one algebraic
continuous map, so as above it preserves really algebraically independent
dense sets over $\overline{c}$.

In particular, if there is a dense set of really algebraically independent
points in $U_{Y}$, then there is a corresponding dense set of really
algebraically independent points in $\tilde{\psi}_{W}(U_{Y})$. However,
this set is exactly $k\cdot\psi_{W}^{-1}(U_{Y})$. Since $\psi_{W}^{-1}(U_{Y})$
has complement of $o$-minimal dimension strictly smaller than $W$,
and the local dimension of $W$ is always $\odim(W)=2\cdot\dim(W)$,
we have that $\psi_{W}^{-1}(U_{Y})$ is dense in $W$; hence, its
multiple $k\cdot\psi_{W}^{-1}(U_{Y})$ is dense in $k\cdot W$.

In particular, the image of the solutions through $\tilde{\psi}_{W}$
is dense in $k\cdot W$. But for all $W\in\Rc(V)$ there is a $W'$
such that $k\cdot W'=W$; therefore, if all the open sets in the subsystem
$\Uc$ contain a dense set of really algebraically independent solutions,
the same is true for all the varieties in $\Rc(V)$.\end{proof}

%% file: Proof.tex
\section{\label{sec:proof}$K_{E_{|K|}}$ is a Zilber field}

We finally proceed to the full verification that all the axioms listed
in \prettyref{sec:defaxioms} are satisfied in the $E$-field $K_{E_{|K|}}$
produced in \prettyref{sec:cons}.
\begin{prop}
\label{prop:obvious}$K_{E_{|K|}}$ satisfies (ACF\textsubscript{0}),
(E), (LOG), (SEC).\end{prop}
\begin{proof}
The axiom (ACF\textsubscript{0}) is trivially true, as $K$ is algebraically
closed and has characteristic $0$.

At each substep, we define $E$ on some new elements $\Q$-linearly
independent from the previous domain; hence, the function is well-defined
and it satisfies the equation $E(x+y)=E(x)\cdot E(y)$. Moreover,
since we run over the whole enumeration $\{\alpha_{j}\}$ of $R\cup iR$,
and $R+iR=K$, we have $\dom(E_{|K|})=K$. Therefore $K_{E_{|K|}}$
satisfies (E).

Similarly, we run over the whole enumeration $\{\beta_{j}\}$ of $R_{>0}\cup\Sb^{1}$,
and $K^{\times}$ is $R_{>0}\cdot\Sb^{1}$, so $E_{|K|}(K)=K^{\times}$,
i.e., $K_{E_{|K|}}$ satisfies (LOG).

Every Kummer-generic simple variety $V$ appears in the sequence $\{V_{j}\}$,
so for some finite $\overline{c}$ there are infinitely many $\overline{z}$
such that $\zEEz{E_{|K|}}{\overline{z}}\in V_{j}$, and they are algebraically
independent over $\overline{c}$. By \prettyref{fact:sec-is-minus},
$K_{E_{|K|}}$ satisfies (SEC).
\end{proof}
Note that moreover the set of the solutions of Kummer-generic simple
varieties is dense, so that axiom (DEN) of \prettyref{sec:involuntary}
is satisfied as well.
\begin{prop}
\label{prop:std}Let $K_{E}\subset K_{E'}$ be a partial $E$-field
extension produced by a finite operation. If $\im(E)$ contains all
the roots of unity, then $\ker(E)=\ker(E')$.

In particular, $\ker(E_{j})=i\omega\Z$ for all $j\leq|K|$, and $K_{E_{|K|}}$
satisfies (STD) with $\ker(E_{|K|})=i\omega\Z$.\end{prop}
\begin{proof}
It is clear that $\ker(E)\subset\ker(E')$, so it is sufficient to
prove the other inclusion.

\textbf{\textsc{domain.}} If $\alpha\in D$, the conclusion is clear,
so we may assume that $\alpha\notin D$. Let us suppose that in $\dom(E')$
there is an element $x$ such that $E'(x)=1$. By definition of $E'$,
it means that there is a $z\in D$ and a rational number $\frac{p}{q}$
such that $x=z+\frac{p}{q}\alpha$ and $E'(x)=E(z)\cdot\beta^{p/q}=1$.
This means that $\beta^{p}=E(z)^{-q}\in E(D)$, but $\beta$ is transcendental
over $E(D)$, so we must have $p=0$. This implies $x=z\in D$, and
in turn $x\in\ker(E)$.

\textbf{\textsc{image.}} If $\beta\in E(D)$, the conclusion is again
clear, so we may assume that $\beta\notin E(D)$. As above, if $x\in\dom(E')$
is such that $E'(x)=1$, then $x=z+\frac{p}{q}\alpha$ and $E'(x)=E(z)\cdot\beta^{p/q}=1$
for some $z\in D$ and some rational number $p/q$. However, since
$\beta$ is not a root of unity, then $z\neq0$; if $p\neq0$ this
implies $\beta=E(-\frac{q}{p}z)$, i.e., $\beta\in E(D)$, a contradiction.
Therefore, $p=0$. This implies again $x=z\in D$, so $x\in\ker(E)$.

\textbf{\textsc{sol}}\textbf{.} By \prettyref{thm:rotund}, if $V$
is absolutely free, then $\check{V}$ is absolutely free. By \prettyref{fact:rotund-is-strong},
$\ker(E)=\ker(E')$.

\textbf{\textsc{roots.}} This operation is just a sequence of applications
of \textsc{sol}, hence $\ker(E)=\ker(E')$.

Now, the kernel of $E_{0}$ is exactly $i\omega\Z$. Therefore, $\ker(E_{j})=\ker(E_{j+1})$
for all $j<|K|$. The limit case is trivial.

\textbf{\textsc{limit.}} Since the equality holds at each finite operation,
it must hold at the limit as well: $\ker(E')=\bigcup_{k'<j}\ker(E_{k'})=\ker(E_{k})$
for any $k<j$.

In particular, for all $j\leq|K|$, $\ker(E_{j})=\ker(E_{0})=i\omega\Z$.\end{proof}
\begin{prop}
\label{prop:sp}Let $K_{E}\subset K_{E'}$ be a partial $E$-field
extension produced by a finite operation. Then $K_{E}\leq K_{E'}$.

In particular, $K_{E_{j}}$ satisfies (SP) for all $j\leq|K|$, and
$K_{E_{|K|}}$ satisfies (SP).\end{prop}
\begin{proof}
Let us check the operations one by one.

\textbf{\textsc{domain.}} If $\alpha\in D$, the conclusion is clear,
so we may assume that $\alpha\notin D$. But this holds:
\[
\td(\alpha,E'(\alpha)/D\cup E(D))\geq\td_{F(\alpha)}(\beta)=1=\ld(\alpha/D),
\]
 hence $\delta'(\alpha/D)\geq0$, i.e., $K_{E}\leq K_{E'}$.

\textbf{\textsc{image.}} If $\beta\in E(D)$, the conclusion is again
clear, so we may assume that $\beta\notin E(D)$. As above,
\[
\td(\alpha,E'(\alpha)/D\cup E(D))\geq\td(\alpha/F(\beta))=1=\ld_{D}(\alpha),
\]
 hence $K_{E}\leq K_{E'}$.

\textbf{\textsc{sol.}} By \prettyref{thm:rotund}, if $V$ is rotund,
then $\check{V}$ is rotund. By \prettyref{fact:rotund-is-strong},
$K_{E}\leq K_{E'}$.

\textbf{\textsc{roots.}} This operation is just a sequence of applications
of \textsc{sol}; but the limit of strong extensions is a strong extension,
hence $K_{E}\leq K_{E'}$.

Now, $K_{E_{0}}$ satisfies (SP). Again, the limit case is trivial.

\textbf{\textsc{limit.}} As before, the limit of strong extensions
is a strong extension, hence $K_{E_{k}}\leq K_{E'}$ for any $k<j$.

This implies that $K_{E_{0}}\leq K_{E_{j}}$, so by \prettyref{fact:strong-keeps-sc}
$K_{E_{j}}$ satisfies (SP).\end{proof}
\begin{prop}
\label{prop:invo}For all $j\geq-1$, $\sigma\circ E_{j}=E_{j}\circ\sigma$.\end{prop}
\begin{proof}
Indeed, by construction we have $E_{j}(R)\subset R_{>0}$ and $E_{j}(iR)\subset\Sb^{1}$,
and $\dom(E)$ is always closed under $\sigma$; by \prettyref{prop:commutes},
this implies that $\sigma\circ E_{j}=E_{j}\circ\sigma$.
\end{proof}
Proving (CCP) is much more difficult, at least in the uncountable
case. Let $D_{j}$ be the domain of the function $E_{j}$, and $F_{j}$
the field generated by $D_{j},E(D_{j})$, for $j\leq|K|$. We note
that, under certain conditions, only the operations \textsc{sol} and
\textsc{roots} can introduce new solutions to a given variety.
\begin{lem}
\label{lem:domain-image-no-new-sol}Let $K_{E}\leq K_{E'}$ be a partial
$E$-field extension and let $X(\overline{c})\subset\G^{n}$ be a
perfectly rotund variety $E$-defined over $\overline{c}\subset\acl(\dom(E),\im(E))$.

If the extension is the result of \textsc{domain} or \textsc{image},
then the generic solutions of $X(\overline{c})$ in $K_{E'}$ are
all contained in $K_{E}$.\end{lem}
\begin{proof}
Clearly, we may assume that the extension is not trivial. For brevity,
let $D:=\dom(E)$ and $D':=\dom(E')$, and let us call $\alpha$,
$\beta$ the numbers such that $D'=D\oplus\alpha$ and $E'(\alpha)=\beta$.
Let $F$ be the field generated by $D,E(D)$.

Let $\overline{x}$ be a generic solution of $X(\overline{c})$ in
$K_{E'}$. The vector $\overline{x}$ must be of the form $\overline{z}+\alpha\cdot\overline{m}$,
where $\overline{m}$ is a vector in $\Z^{n}$ and $\overline{z}\in D^{n}$.
We want to prove that $\overline{m}=0$. By using a square integer
matrix $M$ of maximum rank, we may transform the solution to one
of the form 
\[
\zEz{M\cdot\overline{z}+\alpha\cdot m\cdot\overline{e}_{1}}\in M\cdot X(\overline{c}),
\]
 where $m$ is some integer and $\overline{e}_{1}$ is the vector
$(1,0,\dots,0)$$ $. The variety $M\cdot X$ is still perfectly rotund.
We distinguish two cases.

If $n\geq2$, let $\overline{e}_{j}$ be the vectors that are $1$
on the $j$-th coordinate, and $0$ on the rest. Let $N$ be the matrix
which is the identity on $\overline{e}_{j}$ for $j>1$, and the zero
map on $\overline{e}_{1}$. Since $M\cdot X$ is perfectly rotund,
then $\dim(N\cdot M\cdot X)=n$, hence the point $\zEz{N\cdot M\cdot\overline{z}}$
has transcendence degree $n$ over $\overline{c}E'(\overline{c})$,
and in particular over $\overline{c}E(\overline{c})$. In particular,
\begin{eqnarray*}
 & \td_{\overline{c}E(\overline{c})}\zEz{N\cdot M\cdot\overline{z}}=\\
 & =\td_{\overline{c}E(\overline{c})}\zEz{M\cdot\overline{z}+\alpha\cdot m\cdot\overline{e}_{1}}.
\end{eqnarray*}
If $m\neq0$, this implies that $\alpha$ and $\beta=E(\alpha)$ are
both algebraic over $\zEz{\overline{z}}\cup\overline{c}E(\overline{c})$,
and in particular over $F(\overline{c}E(\overline{c}))\subset\acl(F)$.
However, we have $\td_{F}(\alpha,\beta)\geq1$, so $m=0$. This implies
that $\overline{m}=0$, as desired.

If $n=1$, then $\dim X(\overline{c})=1$, and the new point is of
the form $z+\alpha\cdot m$. Since the variety is absolutely free,
if $m\neq0$ we have that $z+\alpha\cdot m$ and $E(z)\cdot\beta^{m}$
are both transcendental over $\overline{c}E(\overline{c})$, but interalgebraic
over $\overline{c}E(\overline{c})$; in other words, they are the
solution of an irreducible polynomial over $\overline{c}E(\overline{c})$
where both of them appear. In particular, $\alpha$ and $\beta$ are
interalgebraic over $F$. However, by construction we either have
$\td_{F(\alpha)}(\beta)=1$ or $\td_{F(\beta)}(\alpha)=1$, therefore
$m=0$. Again, this implies that $\overline{m}=0$, as desired.
\end{proof}
Moreover, under similar hypothesis, if the operation \textsc{sol}
applied to $V$ introduces new solutions to $X$, then there is a
strong relation between $V$ and $X$ as semi-algebraic varieties.
\begin{lem}
\label{lem:new-sol-alg-map}Let $K_{E}\leq K_{E'}$ be a partial $E$-field
extension and let $X(\overline{c})\subset\G^{n}$ be a perfectly rotund
variety $E$-defined over $\overline{c}\subset\acl(\dom(E),\im(E))$.
Suppose that the extension is the result of \textsc{sol} applied to
an absolutely irreducible, absolutely free rotund variety $V\subset\G^{m}$
which is $E$-defined over $\acl(\dom(E),\im(E))$.

If $X(\overline{c})$ contains a generic solution in $K_{E'}$ which
is not in $K_{E}$, then there is an integer matrix $M$ of maximum
rank, a vector $\overline{z}\subset\dom(E)$ and a $W\in\Rc(V)$ such
that
\[
M\cdot\check{W}\oplus\zEz{\overline{z}}=X.
\]
\end{lem}
\begin{proof}
Let us suppose that $\intv{\overline{\alpha}}{\overline{\beta}}\in\check{V}$
is the new point we add to the graph of the exponential function.
As before, let $D$ be the domain of the exponential function before
adding the point $\intv{\overline{\alpha}}{\overline{\beta}}$, and
let $D':=\span_{\Q}(D\cup\overline{\alpha})$ be the domain after.
The vector $\overline{x}$ must be of the form $\overline{z}+M\cdot\overline{\alpha}$,
for some matrix $M\in\Mc_{n,2m}(\Q)\setminus\{0\}$, and $\overline{z}\in D^{n}$.
Let $F:=\Q(D,E(D))$.

Suppose that $M$ is an integer matrix. In this case, note that $\overline{x}$
is also a generic solution (over $F$) of 
\[
\zEz{\overline{z}+M\cdot\overline{\alpha}}\in M\cdot\check{V}\oplus\zEz{\overline{z}}
\]

Recall that $\td_{F}(\overline{\alpha},E(\overline{\alpha}))=\dim\check{V}=2m$,
and that for any matrix $P$ we have $\td_{F}(P\cdot\overline{\alpha},E(P\cdot\overline{\alpha}))\geq\rank P$.

Now, let $N$ be an square integer matrix of maximum rank such that
the first rows of $N\cdot M$ forms a matrix $Q$ of maximum rank
equal to $\rank M$, and that the remaining rows are zero. Clearly,
$N\cdot\overline{z}+N\cdot M\cdot\overline{\alpha}$ is a generic
solution of $N\cdot X(\overline{c})$ in $K_{E'}$, and in particular
$\zEz{N\cdot\overline{z}+N\cdot M\cdot\overline{\alpha}}$ is a point
of $N\cdot X$ of transcendence degree $n$ over $\overline{c}E(\overline{c})$.

Let $N\cdot\overline{z}=\overline{z}'\overline{z}''$, where $\overline{z}'$
is formed by the first $\rank M$ coordinates and $\overline{z}''$
by the remaining $(n-\rank M)$ ones. Let us suppose that $n>\rank M$.
By simpleness of $N\cdot X(\overline{c})$, we have $\td_{\overline{c}E(\overline{c})}(\overline{z}'',E(\overline{z}''))>(n-\rank M)$.

In particular, we also have $\td_{\overline{c}E(\overline{c})\overline{z}''E(\overline{z}'')}(\overline{z}'+Q\cdot\overline{\alpha},E(\overline{z}'+Q\cdot\overline{\alpha}))<\rank M$.
However, this contradicts the fact that $\td_{F}(Q\cdot\overline{\alpha},Q\cdot E(\overline{\alpha}))\geq\rank Q=\rank M$.
This implies that $n=\rank M$.

So $\zEz{\overline{z}+M\cdot\overline{\alpha}}$ is a generic point
of $X$ \emph{over $F$}, while it is also a generic point of $M\cdot\check{V}\oplus\zEz{\overline{z}}$
over $F$. This immediately implies the equality $M\cdot\check{V}\oplus\zEz{\overline{z}}=X$.

If $M$ is not an integer matrix, let $l$ be a non-zero integer such
that $l\cdot M$ is an integer matrix and let $W\in\Rc(V)$ be the
variety such that $l\cdot W=V$ and $\zEz{\frac{1}{l}\overline{\alpha}}\in W$;
the above argument applied to $\frac{1}{l}\overline{\alpha}$, $W$,
$l\cdot M$ and $\overline{z}$ implies that 
\[
(l\cdot M)\cdot\check{W}\oplus\zEz{\overline{z}}=X,
\]
 as desired.\end{proof}
\begin{cor}
\label{cor:new-sol-densely}Under the hypothesis of \prettyref{lem:new-sol-alg-map},
if $V$ and $X$ are Kummer-generic, and $V$ is simple, there is
an open subsystem $\Uc$ of $\Rc(X)$ such that $\Rc(V)$ is densely
solved in $K_{E}$ if and only if $\Rc(X)$ is densely solved in $K_{E}$
w.r.t.\ $\Uc$.\end{cor}
\begin{proof}
By \prettyref{lem:new-sol-alg-map} we have the equality
\[
M\cdot\check{W}\oplus\zEz{\overline{z}}=X
\]
 for some $W\in\Rc(V)$. This implies $\dim M\cdot\check{W}=\rank M$;
by \prettyref{thm:rotund}, $M$ is either the form $(N|P)$, with
both $N$ and $P$ invertible of rank $n=m$, or $M$ itself is invertible.
The latter case can be excluded, as it would imply that $\check{W}$
is simple, a contradiction. Moreover, $W$, and in turn $V$, must
be perfectly rotund.

In particular, by \prettyref{prop:dense-solutions-in-roots} there
is an open subsystem $\Uc$ such that $\Rc(V)$ is densely solved
in $K_{E}$ if and only if $\Rc(X)$ is densely solved in $K_{E}$
w.r.t.\ $\Uc$.\end{proof}
\begin{prop}
\label{prop:ccp}For all $j\leq|K|$, $K_{E_{j}}$ satisfies (CCP).
In particular, $K_{E_{|K|}}$ satisfies (CCP).\end{prop}
\begin{proof}
It is clear that $K_{E_{0}}$ satisfies (CCP), as $\dom(E_{0})=\omega\Q$
is countable. We want to prove that (CCP) holds on $K_{E_{j}}$ by
induction. If $K$ is countable, (CCP) is trivially true, so we shall
assume that $K$ is uncountable.

First of all, for all $j$ there is a finite or countable set $B_{j}$
such that $D_{j+1}=\span_{\Q}(D_{j}\cup B_{j})$. By \prettyref{prop:pres-ccp},
if $K_{E_{j}}$ satisfies (CCP), then $K_{E_{j+1}}$ satisfies (CCP).
We claim the induction works also at limit ordinals.

Let $j$ be a limit ordinal such that for all $k<j$, $K_{E_{k}}$
satisfies (CCP). By \prettyref{prop:pres-ccp-field}, in order to
prove (CCP) for $K_{E_{j}}$ it is sufficient to verify that for any
perfectly rotund variety $E_{j}$-defined over $D_{j}$, the number
of generic solutions is at most countable. As mentioned in the proof
of \prettyref{fact:sec-is-minus}, any such variety can be written
as $q\cdot(W_{1}\cup\dots\cup W_{m})$ with $W_{1},\dots,W_{m}$ Kummer-generic.
Therefore, we may restrict to Kummer-generic varieties $E_{j}$-defined
over $\acl(F_{j})$.

Let $X(\overline{c})\subset\G^{n}$ be a Kummer-generic perfectly
rotund variety $E_{j}$-defined over $\acl(F_{j})$. If $\overline{x}\in D_{j}^{n}$
is a generic solution of $X(\overline{c})$ in $K_{E_{j}}$, then
there is a smallest $k$ such that $\overline{x}\in D_{k+1}^{n}$
and $\overline{c}E_{j}(\overline{c})=\overline{c}E_{k}(\overline{c})\subset\acl(F_{k})$.
Let $\Lambda\subset j$ be the subset composed by the ordinals $k$;
the solutions of $X(\overline{c})$ in $K_{E_{j}}$ will be the union
of all the solutions in each $K_{E_{k+1}}$ for $k\in\Lambda$. We
claim that $\Lambda$ has countable cofinality, i.e., that the above
union can be written as a countable union.

Let $\overline{x}$ be a new generic solution of $X(\overline{c})$
contained in $D_{k+1}\setminus D_{k}$ for some $k\in\Lambda$. By
\prettyref{lem:domain-image-no-new-sol}, the solution $\overline{x}$
must have appeared during the substep~\prettyref{enu:rsol}. Since
we are working with $K$ uncountable, we are dealing with the operation
\textsc{roots}, which is a sequence of operations \textsc{sol.}

In order to account for the fact that at each step $\lambda\in\Lambda$
we apply the operation \textsc{sol} $\omega$ times, we introduce
an extra index in the enumeration: we consider the set $\Lambda\times\omega$
with the lexicographic order, so that $(k,l)$ correspond to the $l$-th
application of \textsc{sol} during the $k$-th application of \textsc{roots}.
We call $\Lambda'\subset\Lambda\times\omega$ the set of indices corresponding
to the applications of \textsc{sol} that actually produce new solutions.
With a little abuse of notation, we write $E_{k'}$, with $k'\in\Lambda'$,
to denote the exponential function to which \textsc{sol} is being
applied at the $k'$-th application. Note that if $\pi:\Lambda\times\omega\to\Lambda$,
then clearly $\pi(\Lambda')=\Lambda$.

By \prettyref{cor:new-sol-densely}, for each $k'\in\lambda'$, we
may pick one subsystem $\Uc_{k'}$ such that $\Rc(V_{k'})$ is densely
solved\emph{ }$K_{E_{k'}}$ if and only if $\Rc(X)$ is densely solved
in $K_{E_{k'}}$ w.r.t.\ $\Uc_{k'}$.

By second-countability of the order topology, we can extract an at
most countable subsequence $\{\Uc_{k_{p}'}\}_{p<f}$, for some $f\leq\omega$,
such that for each $Y\in\Rc(X)$, the union $\bigcup_{k'\in\Lambda'}\Uc_{k'}(Y)$
is covered by $\bigcup_{p<f}\Uc_{k_{p}'}(Y)$. We claim that the projection
$(\pi(k_{p}'))_{p<f}$ is cofinal in $\Lambda$.

Indeed, let $\tilde{k}\in\Lambda$ be an index such that $\tilde{k}>\pi(k_{p}')$
for all $p<f$. There exists an integer $l$ such that $\tilde{k}':=(\tilde{k},l)\in\Lambda'$;
let $V_{\tilde{k}'}$ be the corresponding variety to which we are
applying \textsc{sol}. By the previous argument, $\Rc(V_{\tilde{k}'})$
is densely solved in $K_{E_{\tilde{k}'}}$ if and only if $\Rc(X)$
is densely solved in $K_{E_{\tilde{k}'}}$ w.r.t.\ $\Uc_{\tilde{k}'}$.

However, if $U\subset Y\in\Rc(X)$ is an open set of the subsystem
$\Uc_{\tilde{k}'}$, then $U$ is covered by the union of some open
sets $U_{p}$, with $U_{p}\in\Uc_{k_{p}'}$. By construction, after
the steps $\pi(k_{p}')$, which correspond to full applications of
\textsc{roots}, and therefore in $K_{E_{\tilde{k}'}}$, the open sets
$U_{p}$ contain dense sets of really algebraically independent solutions.
But then the open set $U$ contains a dense set of really algebraically
independent solutions in $K_{E_{\tilde{k}'}}$. This means that the
family $\Rc(X)$ is densely solved in $K_{E_{\tilde{k}'}}$ w.r.t.\ $\Uc_{\tilde{k}'}$,
and in turn $\Rc(V_{\tilde{k}'})$ is densely solved in $K{}_{E_{\tilde{k}'}}$.

But by the definition of the operation \textsc{roots}, this actually
implies that at step $\tilde{k}'$ we \emph{do not} add a new solution
to $V_{\tilde{k}'}$, a contradiction. This proves that $(\pi(k_{p}'))_{p<f}$
must be cofinal in $\Lambda$, and therefore $\Lambda$ has at most
countable cofinality.

Now, the set of the generic solutions of $X(\overline{c})$ contained
in $K_{E_{j}}$ is the union of the solutions contained in $K_{E_{k+1}}$
for $k\in\Lambda$; but its cofinality is at most countable, hence
it is just the union of the solutions contained in $K_{E_{k}}$ for
$k$ running on at most countable subsequence. But this is an at most
countable union of countable sets, hence $X(\overline{c})$ has at
most countably many generic solutions in $K_{E_{j}}$.

By induction, this proves that $K_{E_{j}}$ satisfies (CCP) for all
$j\leq|K|$ .
\end{proof}
We have obtained our proof:
\begin{proof}[Proof of \prettyref{thm:from-invo}]
 By Propositions \ref{prop:obvious}, \ref{prop:std}, \ref{prop:sp}
and \ref{prop:ccp}, $K_{E_{|K|}}$ is a Zilber field, and by \prettyref{prop:invo},
$\sigma$ commutes with $E_{|K|}$, hence it is an involution of $K_{E_{|K|}}$.
Moreover, the kernel of $E_{|K|}$ is exactly $i\omega\Z$.

Taking $E=E_{|K|}$ we obtain the desired result.\end{proof}
\begin{rem}
By construction, each simple variety has a dense set of solutions.
This proves the existence of at least one model of cardinality $2^{\aleph_{0}}$
for \prettyref{thm:invo-dense}, and the existence of the other models
follow by the argument given in \prettyref{sec:involuntary}.
\end{rem}

\begin{rem}
Our construction is potentially abundant: we add ``real generic''
solutions to rotund varieties, and it is not at all clear if a Zilber
field with an involution should always satisfy this condition.\end{rem}

%% file: Further.tex
\section{\label{sec:further}Further results}

The general strategy used in our construction can be useful for showing
some other properties of Zilber fields.

Here we use it to obtain a full classification of which partial $E$-fields
can be embedded into Zilber fields. It is immediate to see that in
order to be embeddable, a partial $E$-field must satisfy (SP), (CCP),
and the kernel must be either trivial or cyclic; we prove that this
is also sufficient. This was already known, but it has never been
explicitly stated. A corollary of this statement is that if Schanuel's
Conjecture is true, then we know at least that $\C_{\exp}$ embeds
into $\B$.

Moreover, we show that if we replace the full axiom (SEC) with a weaker
version, stating that just certain simple varieties have solutions
(namely, curves in $\G^{1}$), then we can create a function $E$
such that $E_{\restriction K^{\sigma}}$ is order-preserving. It has
been proven in \cite{Marker2006} that if Schanuel's Conjecture is
true, then the weaker axiom holds on $\C_{\exp}$ at least for the
case of algebraic coefficients.

If we go further and drop completely axiom (SEC), then we can even
obtain a continuous $E$. The study of these cases shows quite explicitly
where the general construction fails at producing an order-preserving
or a continuous $E$.

The resulting fields embed into all the Zilber fields of equal and
larger cardinality. This is analogous to the result of Shkop \cite{Shkop2011},
who proved that there are several real closed fields inside Zilber
fields such that the restriction of the pseudoexponential function
to the real line is monotone.

\input{Embedding.tex}

\input{Monotone.tex}

%% file: Embedding.tex
\subsection{\label{sec:embedding}Embeddability}

If we drop the involution $\sigma$, the construction of \prettyref{sec:cons}
gives a general method to directly build Zilber fields. As a bonus,
we can actually build a Zilber field extending a given partial exponential
function satisfying certain properties. A direct construction of this
kind is studied in \cite{Kirby2009a}, but without the relevant proofs
for the uncountable case.

Our approach is indeed similar to the one of \cite{Kirby2009a}, except
that we do things in a slightly different order.
\begin{thm}
\label{thm:embeds}Let $K_{E}$ be a partial $E$-field satisfying
(SP) and (CCP).

If $\ker(E)=\{0\}$ or $\ker(E)=i\omega\Z$ for some $\omega\in K^{\times}$,
then there is a strong embedding $K_{E}\leq L_{E'}$ into a Zilber
field $L_{E'}$.
\end{thm}
The converse is clearly true: if a partial $E$-field is (strongly)
embedded into a Zilber field, then it satisfies (SP) and (CCP), and
its kernel is either trivial or cyclic.

As $\C_{\exp}$ satisfies (CCP), this shows that Schanuel's Conjecture
is true if and only if $\C_{\exp}$ embeds into $\B$.

In order to prove \prettyref{thm:embeds} we proceed as in \prettyref{sec:cons}:
we enlarge the field $K$ to some bigger field $L$ and we define
$E'$ by back-and-forth, extending $E$.

First of all, we choose an algebraically closed field $L\supset K$
whose cardinality is strictly greater than the cardinality of $K$.
If $\ker(E)$ is non-trivial, we fix $E_{-1}:=E$. Otherwise, we choose
an arbitrary $\omega\in L$ transcendental over $K$, a coherent system
of roots of unity $(\zeta_{q})_{q\in\N^{\times}}$, and we define
$E_{-1}(x+\frac{p}{q}\omega)=E(x)\cdot\zeta_{q}^{p}$, with $x\in\dom(E)$.

$L_{E_{-1}}$ is our new base step. The rest of the construction of
the sequence $(E_{j})_{j\leq|L|}$ is the same as \prettyref{sec:cons},
just by dropping all the references to $\sigma$, with the obvious
changes in the enumerations, in the operations and in the construction.
Here we show explicitly just the most sensible changes:
\begin{description}
\item [{\textsc{sol'}}] We start with an absolutely irreducible variety
$V(\overline{c})\subset\G^{n}$, where $\overline{c}$ is a subset
of $K$.\\
We choose a point $((\alpha_{1},\beta_{1}),\dots,(\alpha_{n},\beta_{n}))\in V$
such that
\[
\td_{F(\overline{c}E(\overline{c}))}(\alpha_{1},\gamma_{1},\beta_{1},\delta_{1},\dots)=2\dim V.
\]
 We fix an arbitrary system of roots $\beta_{j}^{1/q}$ of $\beta_{j}$
and we define
\[
E'\left(z+\frac{p_{1}}{q_{1}}\alpha_{1}+\dots+\frac{p_{n}}{q_{n}}\alpha_{n}\right):=E(z)\cdot\beta_{1}^{p_{1}/q_{1}}\cdot\dots\cdot\beta_{n}^{p_{n}/q_{n}}.
\]

\item [{\textsc{roots'}}] We start with an absolutely rotund variety $V(\overline{c})\subset\G^{n}$,
where $\overline{c}$ is a finite subset of $K$.\\
Consider an enumeration $W_{m}(\overline{d}_{m})_{m<\omega}$ of all
the varieties $W_{m}$ of $\Rc(V)$, where $\overline{d}_{m}$ is
a finite subset of $\acl(\overline{c}E(\overline{c}))$ over which
$W$ is $E$-defined. We produce inductively a sequence of partial
exponential functions $E_{m}$, starting with $E_{0}:=E$.\\
Let us suppose that $E_{m-1}$ has been defined. If $W_{m}(\overline{d}_{m})$
has an infinite set of algebraically independent solutions in $K_{E_{m-1}}$,
we define $E_{m}:=E_{m-1}$; otherwise we apply \textsc{sol'} to $W_{m}(\overline{d}_{m})$
over the $E$-field $K_{E_{m-1}}$. The resulting exponential function
will be $E_{m}$.\\
Finally, we define $E':=\bigcup_{m\in\N}E_{m}$.
\end{description}
In other words, we drop any reference to the now absent order topology.

The proof that the final structure $L_{E_{|L|}}$ is a Zilber field
is also quite similar, with another sensible difference in managing
the operation \textsc{roots'}. We sketch the modified proof here.
\begin{proof}[Proof of \prettyref{thm:embeds}]
We need to prove that $L_{E_{|L|}}$ is a Zilber field.

First of all, we observe that $K_{E}\leq L_{E_{-1}}$, and that (CCP)
holds on $L_{E_{-1}}$. Then axioms (ACF\textsubscript{0}), (E),
(LOG), (SEC), (STD) and (SP) can be verified by repeating almost identically
the proofs of Propositions \ref{prop:obvious}, \ref{prop:std} and
\ref{prop:sp}.

In order to verify (CCP), we need to repeat the proof of \prettyref{prop:ccp}
with one essential change: since in\textsc{ roots'} we are adding
generic point to \emph{simple} Kummer-generic varieties rather than
their $\G$-restrictions, we can use the stronger \prettyref{prop:solutions-in-roots}
in place of \prettyref{prop:dense-solutions-in-roots}. This is sufficient
to prove that the set $\Lambda$ has cardinality at most $1$, rather
than proving that it has countable cofinality. This provides again
the induction at limit ordinals.

Therefore, $L_{E_{|L|}}$ is a Zilber field.
\end{proof}
Note that the use of \prettyref{prop:solutions-in-roots} instead
of \prettyref{prop:dense-solutions-in-roots}, and the fact that in
the operation \textsc{roots'} we do not need a dense set of solutions,
let the induction work on cardinalities greater than $2^{\aleph_{0}}$.
Note also that the the operation \textsc{roots'} becomes equivalent
to applying \textsc{sol'} countably many times to $V$; in this way,
we may actually get rid of the operation \textsc{roots'} and just
enumerate each variety countably many times, mimicking the construction
for the countable case.

%% file: Monotone.tex
\subsection{\label{sub:order}A monotone $E_{\restriction K^{\sigma}}$ and a
continuous $E$ (without (SEC))}

Our construction fails at producing an exponential function $E$ such
that $E$ is continuous, and it is not even capable to make the restriction
$E_{\restriction K^{\sigma}}$ increasing.

In order to see why, we present first an example construction where
$E$ is increasing, at the price of dropping the axiom (SEC) from
the final structure. Not everything is lost, however, as we manage
to verify a partial instance of (SEC) which we call ``(1-SEC)''.
This example makes evident where our technique fails in producing
order-preserving exponential functions.

The following axiom is the special instance of (SEC) we manage to
verify.
\begin{enumerate}
\item [(1-SEC)] $1$-dimensional Strong Exponential-algebraic Closure:
for every absolutely free rotund variety $V\subset\G^{1}$ irreducible
over $K$, and every tuple $\overline{c}\in K^{<\omega}$ such that
$V$ is defined over $\overline{c}$, there is a generic solution
of $V(\overline{c})$.
\end{enumerate}
It is known that if Schanuel's Conjecture is true, then (1-SEC) holds
on $\C_{\exp}$ at least in the case of $V$ defined over $\oQ$ \cite{Marker2006}.
It is not known if Schanuel's Conjecture implies also (SEC) on $\C_{\exp}$.

The construction, with some adaptation, yields the following.
\begin{thm}
\label{thm:order}For all saturated algebraically closed fields $K$
of characteristic $0$ there is a function $E:K\to K^{\times}$ and
an involution $\sigma$ commuting with $E$ such that $K_{E}$ satisfies
(E), (LOG), (STD), (SP), (1-SEC) and (CCP), and $E_{\restriction K^{\sigma}}$
is a monotone function.
\end{thm}
As in \prettyref{sec:embedding}, we need to do some changes to the
construction. Here is just a sketch of the modified procedure.
\begin{enumerate}
\item We start with a large saturated real closed field $R$ (with some
technical adjustment, we may replace saturated with $\kappa$-saturated
for a sufficiently large cardinal $\kappa$, so that the construction
does not depend on set-theoretic assumptions; however, it will be
easier to just take a saturated $R$).
\item In the operations \textsc{domain} and \textsc{image} on $R$, we exploit
the saturation of $R$ to choose the new values so that $E$ is always
a monotone function.
\item We exploit the saturation in \textsc{sol} to keep $E$ order-preserving,
and we drop any mention of the topology, and in particular of density,
in \textsc{sol} and \textsc{roots}, as we did in \prettyref{sec:embedding}.
\end{enumerate}
The fact that we can keep $E$ order-preserving when adding solutions
to simple varieties is guaranteed by the following lemma of which
we only sketch the proof.
\begin{lem}
\label{lem:order-point} Let $V\subset\G^{1}$ be a simple variety
defined over some $\overline{c}$. If $f:R\to R$ is any partial strictly
increasing function with $|\dom(f)|<|R|$, then there is a point $(z,w)\in V$
generic over $\dom(f),\im(f),\overline{c}$ such that $f\cup\{(\Re(z),|w|)\}$
is an increasing function.\end{lem}
\begin{proof}
Let us fix $x_{1},\dots,x_{n}$ points in $\dom(f)$. We claim that
there is a generic point $(z,w)\in V$ such that for some $1\leq j<n$,
$x_{j}<\Re(z)<x_{j+1}$ and $f(x_{j})<|w|<f(x_{j+1})$. These conditions
would form a finitely satisfiable type as $x_{1},\dots,x_{n}$ vary,
and saturation would imply the existence of the desired point.

The set $B:=\bigcup_{j=0}^{n}\{(x,w)\,:\, x_{j}\leq x\leq x_{j+1},\, f(x_{j})\leq w\leq f(x_{j+1})\}$,
where we assume $x_{0}=-\infty$, $x_{n+1}=+\infty$, $f(-\infty)=0$
and $f(+\infty)=+\infty$, definably disconnects the upper half plane.
The image $z(V)$ is $ $$K$ minus finitely many points, and similarly
$w(V)$ is $K^{\times}$ minus finitely many points; hence, the image
of $\Re(z)$ is the whole line $R$, and the image of $|w|$ is $R^{\times}$.
Moreover, the image of the map $\pi:z\mapsto(\Re(z),|w|)$ must be
definably connected.

Clearly, $\pi(V)\cap B$ must contain at least one point of the form
$(x_{j},f(x_{j}))$. We want to prove that $\pi(V)\cap B$ must actually
contain points in the interior of $B$.

We start by proving that the map $\pi:z\mapsto(\Re(z),|w|)$ is finite-to-one.
One may exploits the convergence of the Puiseux series, i.e., the
fact for any $P\in V(\C)$ an equality of the form
\[
w=a_{0}+a_{1}(z-z(P))^{k/d}+O((z-z(P))^{(k+1)/d})
\]
 holds in a sufficiently small neighbourhood of $P$. This shows immediately
that whenever $\Re(z)$ is fixed to $\Re(z(P))$ in a neighbourhood
of $P$, a change in $\Im(z)$ forces (locally) a change in $|w|$,
and actually $|w|$ must locally vary in a proper interval. We skip
the details here.

For the sake of notation, let $\pi_{1}$ and $\pi_{2}$ the two functions
$\Re(\cdot)$ and $|\cdot|$; the above statement says that if $U$
is a small neighbourhood of $P$, then $\pi_{2}$ takes different
values on $U\cap\pi_{1}^{-1}(\Re(z(P)))$. Let us assume by contradiction
that $\pi$ is not finite-to-one, i.e., that there is be a point $(c,d)\in R\times R_{>0}$
such that $\pi^{-1}(c,d)$ is infinite.

The sets $\pi_{1}^{-1}(c)$ and $\pi_{2}^{-1}(d)$ have both dimension
1, and since their intersection is infinite, they must contain a common
line $l$. Moreover, we may assume that the line $l$ is small enough
so that if $P\in l$, there is a neighbourhood $U\ni P$ such that
$U\cap\pi_{1}^{-1}(c)\subset l$. However, this implies that $\pi_{2}$
takes only one value on $U\cap\pi_{1}^{-1}(\Re(z(P)))$, a contradiction.

Now, let us consider two points whose images in $\pi(V)$ lie on different
sides w.r.t.\ $B$. Since $V$ is an algebraic variety over the algebraic
closure of $R$, we can connect the points with a definable continuous
path that does not pass through any of the finitely many points whose
image is of the form $(x_{j},f(x_{j}))$. Therefore, the image of
the path in the upper half plane must pass through the interior of
$B$, as desired.

It is easy to see that there are also generic points in $\pi(V)\cap B$
which lie in the interior of $B$. In particular, the partial type
we are trying to realise is finitely satisfiable; by saturation, there
exists a realisation in $R$.
\end{proof}
Note that once $E_{\restriction R}\cup\{(\Re(z),|w|)\}$ is monotone,
the same is true if we extend the function to the rational multiples
of $\Re(z)$.

Since we do not require the solutions to be dense, the proof of (CCP)
is again different. The following facts shows that the situation is
substantially the same as in \prettyref{sec:embedding}. Let us introduce
an intermediate definition between Definitions \ref{def:comp-solved}
and \ref{def:dense-solved}.
\begin{defn}
A system $\Rc(V)$, with $\overline{c}$ defining $\check{V}$, is
\emph{really completely solved} in $K_{E}$ if for all $W\in\Rc(V)$
there is an infinite set of solutions of $W$ really algebraically
independent over $\acl(\overline{c})$.
\end{defn}
The following stronger version of \prettyref{thm:rotund}, pointed
out by Maurizio Monge, holds in this case.
\begin{prop}
\label{prop:rotund-1-dim}Let $V\subset\G^{1}$ be a simple variety.
If $M\in\Mc_{1,2}(\Z)$ is such that $\dim M\cdot\check{V}=1$, then
$M$ is of the form
\[
M=\left(\begin{array}{c|c}
k & \pm k\end{array}\right)
\]
for some integer $k\in\Z$. In particular, the induced map $V\to M\cdot\check{V}$
is surjective.\end{prop}
\begin{proof}
By \prettyref{thm:rotund}, $M$ must be of the form $M=\left(a|b\right)$
for some non-zero integers $a,b$, and $\dim V=1$. It is sufficient
to verify the claim on $\R$.

Locally, the functions $x+iy$ and $ax+iby$ are holomorphic functions
of $\rho\theta$ and $\rho^{a}\theta^{b}$ respectively. If $\theta=e^{i\vartheta}$,
and if we choose the determination of $\theta^{b/a}$ as $e^{i\vartheta\frac{b}{a}}$,
we may also say that $(x+i\frac{b}{a}y)$ is locally a holomorphic
function of $\rho e^{i\vartheta b/a}$. The Cauchy-Riemann equations
in polar coordinates imply
\[
\frac{\partial y}{\partial\rho}=-\frac{1}{\rho}\cdot\frac{\partial x}{\partial\vartheta}\quad\textrm{and}\quad\frac{\frac{b}{a}\partial y}{\partial\rho}=-\frac{1}{\rho}\cdot\frac{\partial x}{\frac{b}{a}\partial\vartheta}.
\]
 This implies $a^{2}=b^{2}$, i.e., $a=\pm b$, as desired.

Moreover, $M\cdot\check{V}$ is either $k\cdot V$ or $\sigma(k\cdot V)$
for some integer $k$, and the map $V\to M\cdot\check{V}$ is clearly
surjective in both cases.
\end{proof}
Using surjectivity, we can repeat the proofs of Propositions \ref{prop:solutions-in-roots}
and \ref{prop:dense-solutions-in-roots} to obtain the following stronger
result.
\begin{prop}
\label{prop:solutions-surj}Let $V\subset\G^{n}$ be a Kummer-generic
rotund variety and $K_{E}$ a partial $E$-field such that $\sigma\circ E=E\circ\sigma$.
Let $N,P\in\Mc_{n,n}(\Z)$ be two square integer matrix of maximum
rank, $\overline{z}\in\dom(E)^{n}$, and $X:=\left(N|P\right)\cdot\check{V}\oplus\zEz{\overline{z}}$.

For $W\in\Rc(V)$, let $X_{W}\in\Rc(X)$ be the variety $(N|P)\cdot\check{W}\oplus\zEz{\frac{1}{q}\overline{z}}$
for some $q\in\Z^{\times}$. If $X$ is Kummer-generic, and for all
$W$ the induced map $W\to X_{W}$ is surjective, then $\Rc(V)$ is
really completely solved if and only if $\Rc(X)$ is.
\end{prop}
Hence, the proof is much more similar to the one of \prettyref{sec:embedding}.
\begin{proof}[Proof of \prettyref{thm:order}]
 The proof is the same as the one of \prettyref{thm:embeds}, where
we replace the use of \prettyref{prop:solutions-in-roots} with \prettyref{prop:solutions-surj}.
In particular, we prove again that $\Lambda$ has cardinality $1$
rather than having countable cofinality.

After having produced a Zilber field $L_{E}$ on the algebraic closure
of $R$, we reduce the cardinality by taking the closure of some sets
of exponentially-algebraically independent elements of $R$.
\end{proof}
Note that since the resulting $K_{E}$ is an $E$-field satisfying
(STD), (SP) and (CCP), by \prettyref{sec:embedding} it embeds into
a Zilber field.\\

If we drop (1-SEC) as well, it is easy to produce a \emph{continuous}
function $E$.
\begin{thm}
\label{thm:continuous}For all saturated algebraically closed fields
$K$ of characteristic $0$ there is a function $E:K\to K^{\times}$
and an involution $\sigma$ commuting with $E$ such that $K_{E}$
satisfies (E), (LOG), (STD), (SP), and (CCP), and $E$ is a continuous
function with respect to the topology induced by $\sigma$.
\end{thm}
The trick is again to start with a saturated real closed field $R$,
and exploit the saturation in the operations \textsc{domain} and \textsc{image}:
we make sure at once that $E:K^{\sigma}\to K^{\sigma}$ is monotone
and that $E:[0,\omega)\to\Sb^{1}(K)$ moves `counterclockwise'.
This is sufficient to obtain that $E$ is continuous.

As there is no operation \textsc{roots} in this case, the axioms (SP)
and (CCP) are even easier to verify. Again, the resulting structures
embed into Zilber fields.

All of this shows quite well the two obstructions that our method
is not able to overcome: first of all, if $V\subset\G^{n}$, it is
not always true that the topological dimension of $\left(\begin{array}{c|c}
\mathrm{Id} & 0\end{array}\right)\cdot\check{V}$ is $2n$, so the argument of \prettyref{lem:order-point} does not
work in the general case, and we cannot be sure that we are always
able to produce an order-preserving exponential function. This prevents
us from getting continuity as well.

However, we have also seen that if we discover that the map $V\to M\cdot\check{V}$
is surjective when $\dim V>1$ (for example, $M$ could be actually
forced to be of the form $(N|\pm N)$, as in \prettyref{prop:rotund-1-dim}),
then \prettyref{prop:solutions-surj} would apply in all cases; hence,
we would not need the density arguments and the second countability,
as the work of \prettyref{sec:embedding} would be sufficient to get
(CCP) without further complications. In particular, it would be possible
to find involutions on Zilber fields of arbitrary cardinalities, using
arbitrary real closed fields, and we would find models outside of
the class described in \prettyref{sec:involuntary}. Work is in progress
about finding such a generalisation.